\documentclass{amsart}
\usepackage{graphicx} 
\usepackage{amsmath, amssymb, amsthm}
\usepackage{hyperref}
\usepackage{enumitem}
\usepackage{inputenc}
\usepackage{biblatex}

\usepackage{bbm}
\usepackage{pdfpages}
\usepackage{tcolorbox}
\usepackage{esint} 
\usepackage{amsmath,amssymb,amsthm}

\theoremstyle{plain}

\newtheorem{question}{Question}
\newtheorem{conjecture}{Conjecture}

\theoremstyle{definition}

\newtheorem*{acknowledgement}{Acknowledgement}

\newtheorem*{first key estimate}{First key estimate}
\newtheorem*{second key estimate}{Second key estimate}

\newcommand{\bbS}{\mathbb{S}}
\newcommand{\R}{\mathbb{R}}

\newcommand{\Z}{\mathbb{Z}}

\newcommand{\cC}{\mathcal{C}}
\newcommand{\cD}{\mathcal{D}}
\newcommand{\cE}{\mathcal{E}}

\newcommand{\cG}{\mathcal{G}}

\newcommand{\cR}{\mathcal{R}}
\newcommand{\cS}{\mathcal{S}}

\newcommand{\supp}{\mathrm{supp}\,}

\newtheorem{theorem}{Theorem}[section]
\newtheorem{lemma}[theorem]{Lemma}
\newtheorem{remark}[theorem]{Remark}
\newtheorem{definition}[theorem]{Definition}
\newtheorem{corollary}[theorem]{Corollary}
\newtheorem{proposition}[theorem]{Proposition}

\numberwithin{equation}{section}

\title{Euler-Lagrange Flow Monotonicity for the Radon Transform}
\author{Yannis Galanos}
\date{June 2026}
\address{Yannis Galanos, School of Mathematics and Maxwell Institute for Mathematical Sciences, University of Edinburgh, James Clerk Maxwell Building, Peter Guthrie Tait Road, King's Buildings, Edinburgh EH9 3FD, United Kingdom}
\email{I.Galanos@sms.ed.ac.uk}

\bibliography{references}

\begin{document}

\maketitle

\begin{abstract}
    We present a new gradient-like flow that targets critical points of the inequality for the planar Radon transform $\cR$ proven by Oberlin and Stein \cite{Oberlin-Stein}, partially answering a question of Christ, who gave a complete characterisation of the extremisers of the inequality via rearrangement methods. We show that for non-negative radial initial data near the extremiser manifold, solutions to the flow exist globally and there exists a rescaled subsequence of the flow that converges to a critical point of the Radon functional $\|\cR f\|_3/\|f\|_{3/2}$ locally in $L^{3/2}(\R^2)$. We also obtain a quantitative result for the rate of decay of critical points, which complements a result of Christ \cite{Christ} in the special case of radial functions. We prove a novel localised smoothing estimate for the operator $\cR^*[(\cR\cdot)^2]$ restricted to radial functions.  This approach is complementary to that of Christ and provides a robust framework for studying extremal problems associated to Radon-like operators.
\end{abstract}

\section{A Variational Perspective on Radon Extremisers}
Characterising the extremisers of the classical Radon transform $L^{3/2}(\R^2)\to L^3(\R^2)$ inequality of Oberlin and Stein \cite{Oberlin-Stein} is a problem with rich history that goes back to Christ \cite{Christ}. Christ \cite{Christ} characterised all extremisers of the Radon transform in $\R^n$, using rearrangement inequalities and Steiner symmetrisation. These techniques utilise the rich symmetry structure of the Radon transform, which includes all affine symmetries and ellipsoidal symmetries (the latter are not relevant in the context of the present work and thus will not be mentioned further; we refer the interested reader to \cite{Christ} for details). Christ also posed the question of whether the extremal problem for the Radon transform can be studied using variational methods such as the ones used by Carlen, Carrillo and Loss \cite{Carlen-Carrillo-Loss} to study the extremal problem for the Hardy--Littlewood--Sobolev inequality. In fact, the two extremal problems are related.\\\\
In this paper, we give a partial answer to Christ's question. We can tackle the extremal problem for the classical Radon transform $\cR$ in $\R^2$, by constructing a flow that targets critical points of the Radon functional $f\mapsto\|\cR f\|_3/\|f\|_{3/2}$. \textit{However}, the variational mechanism we present is very different from the one introduced in \cite{Carlen-Carrillo-Loss}. In fact, the Radon functional does not have an obvious monotonicity property along the fast diffusion equation which was used in \cite{Carlen-Carrillo-Loss}. Instead, we introduce a gradient-like flow that is specifically designed to approach critical points of the Radon functional $f\mapsto\|\cR f\|_3/\|f\|_{3/2}$. We show that the associated initial value problem is globally well-posed and that for non-negative radial initial data sufficiently close to the extremiser manifold, after suitable rescaling it converges subsequentially to a non-negative radial critical point of $f\mapsto\|\cR f\|_3/\|f\|_{3/2}$. We also prove that every radial critical point of the latter is asymptotically equivalent to the unique radial extremiser.\\\\
Let us recall the classical definition of the Radon transform:
$$\mathcal{R}f(\alpha,s):=\int_{x\cdot\alpha=s}f(x)d\sigma(x),$$
for $\alpha\in\bbS^1, s\in\R$, where $d\sigma(x)$ is the surface measure on the line $x\cdot\alpha=s$. We will also consider the adjoint of $\cR$, given by
\[\cR^*g(x)=\int_{\bbS^1} g(\alpha,x\cdot\alpha)\ d\alpha, \]
where $g:\bbS^1\times\R\to \R$ and $d\alpha$ is the Hausdorff measure on $\bbS^1$. Oberlin and Stein \cite{Oberlin-Stein} proved that there exists a finite constant $C>0$ such that
\begin{equation}\label{eq:Oberlin-Stein}
    \|\cR f\|_{L^3(\bbS^1\times\R)}\leq C\|f\|_{L^{3/2}(\R^2)},
\end{equation}
for all $f\in L^{3/2}(\R^2)$. Let $\|\cR\|$ be the optimal constant in \eqref{eq:Oberlin-Stein}. We define the manifold of extremisers of $\cR$ as 
$$\mathcal{E}:=\left\{ f\in L^{3/2}(\R^2):\ \frac{\|\cR f\|_3}{\|f\|_{3/2}}=\|\cR\| \right\}. $$
Christ \cite{Christ-Extremals, Christ-k-planes} proved that $\mathcal{E}$ is non-empty. In a subsequent work \cite{Christ}, Christ characterised the extremal set $\cE$. The following is a special case of his result.
\begin{theorem}[Christ \cite{Christ}, Theorem $1.3$]\label{thm:Christ}
    A non-negative function $f$ is an extremiser if and only if
    \[f(x)=\frac{c}{1+|\phi(x)|^2},\]
    for some constant $c>0$ and invertible affine map $\phi:\R^2\to \R^2$. 
\end{theorem}
Drouot \cite{Drouot} proved a quantitative stability property of the extremiser manifold $\cE$. In a heuristic sense, Drouot's result asserts that if a function $f$ is far from $\cE$, then the ratio ${\|\cR f\|_3}/{\|f\|_{3/2}}$ is far from $\|\cR\|$. To make this result precise, we define the distance of $f$ from $\cE$ as 
\begin{equation}
    d(f,\cE):=\inf_{h\in\cE}\|f-h\|_{3/2}.
\end{equation}
The following holds.
\begin{theorem}[Drouot \cite{Drouot}]\label{thm:Drouot}
    There exists a constant $c_\cR>0$ such that
    \begin{equation}\label{eq:Drouot}
        \frac{\|\cR f\|_3}{\|f\|_{3/2}}\leq \|\cR\|\left[1-c_\cR \left(\frac{d(f,\cE)}{\|f\|_{3/2}}\right)^2\right],
    \end{equation}
    for all $0\neq f\in L^{3/2}(\R^2)$.
\end{theorem}
Extremisers of the nequality \eqref{eq:Oberlin-Stein} have similar form to their counterparts for the Hardy--Littlewood--Sobolev inequality in $\R^n$,
\begin{equation}\label{HLS}
     \int_{\R^n}\int_{\R^n}\frac{f(x)g(y)}{|x-y|^\lambda}\,dx\,dy \leq C_S\,\|f\|_{L^{p}(\R^n)}\|g\|_{L^{p}(\R^n)},
\end{equation}
for all $0<\lambda<n$, $p=2n/(2n-\lambda)$ and non-negative functions in $L^p(\R^n)$, where $C_S(n,\lambda)$ denotes the optimal constant. Carlen, Carrillo and Loss \cite{Carlen-Carrillo-Loss} studied the extremal problem for \eqref{HLS} via a modified heat-flow approach. They showed that the Hardy--Littlewood--Sobolev functional in $\R^n$ for $n\geq 3,\,\lambda=n-2$,
\begin{equation}\label{HLS functional}
    \mathrm{HLS}(u) := C_S\|u\|_{2n/(n+2)}^2 - \frac{1}{(n-2)|\bbS^{n-1}|}\int_{\R^n}\int_{\R^n}\frac{u(x)u(y)}{|x-y|^{n-2}}dx\,dy,
\end{equation}
is non-increasing along the fast diffusion flow
\begin{equation}\label{fast diffusion}
    u_t = \Delta(u^m),
\end{equation}
when $m=n/(n+2)$, for initial data in $L^{2n/(n+2)}(\R^n)$ with sufficiently fast polynomial decay. In fact, they showed that if $u(t)$ is the solution of \eqref{fast diffusion} at time $t$, then for a suitable scaled version $v(t)$ of $u(t)$, one has
\[\lim_{t\to \infty}\mathrm{HLS}(v(t))=0.\]
They then show that the renormalised solution $v(t)$ converges to a Barenblatt profile
\[h_{c,s,x_0}(x):=\frac{c}{\left(1+|\frac{x-x_0}{s}|^2\right)^{(n+2)/2}},\]
where $c,s>0$ are constants and $x_0\in \R^n$. Their result immediately implies that the only non-negative extremisers of \eqref{HLS} are the profiles $h_{c,s,x_0}.$ \\\\
While Theorem \ref{thm:Christ} completely characterises the extremisers of \eqref{eq:Oberlin-Stein}, Christ remarks that the methods developed in \cite{Christ} do not suffice to identify all critical points of the functional 
\begin{equation}\label{eq:Phi}
    \Phi(f):=\frac{\|\cR f\|_3}{\|f\|_{3/2}}.
\end{equation} 
Before proceeding, we first need to make the notion of critical point precise in this setting:
\begin{definition}
    A function $0\neq f\in L^{3/2}(\R^2)$ is called a critical point of $\Phi$ iff 
    \begin{equation}
        \frac{d}{d\varepsilon}\Bigg|_{\varepsilon=0}\Phi(f+\varepsilon \eta)=0,
    \end{equation}
    for all smooth $\eta\in L^{3/2}(\R^2)$ such that $f+\varepsilon\eta\neq 0$. The quantity $\frac{d}{d\varepsilon}\Big|_{\varepsilon=0}\Phi(f+\varepsilon\eta)$ is equal to the directional Fr\'echet derivative $\nabla_\eta\Phi(f)$ of $\Phi$.
\end{definition}
Christ \cite{Christ} asked whether the flow approach in \cite{Carlen-Carrillo-Loss} can be adapted to characterise the extremisers and critical points of $\Phi$. In this work, we propose a new non-linear gradient flow $f(t):=T(t)f_0$, such that the Radon functional
\[
    \Phi(f(t)) := \frac{\|\cR\, T(t)f_0\|_{L^3(\bbS^1\times\R)}}{\|T(t)f_0\|_{L^{3/2}(\R^2)}}
\]
is non-decreasing and whose equilibrium states are critical points of $\Phi$, given non-negative radial initial data $f_0\in L^{3/2}(\R^2)$. To do so, we prove the following non-explicit characterisation of critical points of $\Phi$ in Section \ref{section:critical points}.
\begin{proposition}\label{prop:extremisers}
    For any $f\in L^{3/2}(\R^2)$ and $\lambda\in \R$, let
    \[Af:= \cR^*\left(|\cR f|\,\cR f\right), \qquad  \cG(f,\lambda):= (Af)^2-\lambda |f|. \]
    A function $f\in L^{3/2}(\R^2)$ is a critical point of $\Phi $ if and only if there exists  $\lambda\in\R$ such that $\cG(f,\lambda)=0$.
\end{proposition}
The equation $\cG(f,\lambda)=0$ is known as the Euler--Lagrange equation for $\Phi$, as its solutions are precisely the solutions of the constrained optimisation problem
\[\nabla\Phi(f)=0,\ \mbox{ subject to }\ \|f\|_{3/2}=\mbox{ constant}.\]
This characterisation of critical points of $\Phi$ naturally gives rise to the following gradient flow (ascent) for $\Phi$:
\begin{gather}\label{eq:EL}
    \partial_t f(x,t)= Af(t)(x)^2-\lambda(f(t))f(x,t),
\end{gather}
where $\lambda(f(t))$ is a Lagrange multiplier chosen such that $\|f(t)\|_{3/2}$ is constant for all times $t$ for which the flow exists. In fact, we will see that there is a unique such choice, namely
\begin{equation}\label{eq:lambda}
    \lambda(f):=\frac{\int |f|^{1/2} (Af)^2\mathrm{sgn}(f)}{\int |f|^{3/2}}.
\end{equation}
The main result of this work is the following:
\begin{theorem}\label{thm:EL main}
    For any $f_0\in L^{3/2}(\R^2)$, the initial-value problem \eqref{eq:EL} has a unique solution $f\in C^1\left([0,\infty), L^{3/2}(\R^2)\right)$ given by
    \begin{equation}\label{eq:Duhamel}
        f(x,t)= e^{-\Lambda(t)} f_0(x)+\int_0^t e^{-[\Lambda(t)-\Lambda(s)]} Af(s)(x)^2 ds,
    \end{equation}
    where $\Lambda(t):=\int_0^t \lambda(f(s))ds$. Furthermore, if $f_0\geq 0$ the following are true:
    \vspace{0.1cm}
    \begin{enumerate}
        \item The functional $\Phi(f(\cdot,t))$ is non-decreasing in $t$.
            \vspace{0.1cm}
        \item Let $\varepsilon\in(0,10^{-2})$. If in addition $f_0$ is radial and
        \begin{align}\label{eq:threshold}
            \Phi(f_0)^3\geq\|\cR\|^3(1-\varepsilon),
        \end{align}
        then there exists a scaling function $\mu:\R_+\to \R_+$ such that $\mu(t)^{4/3}f_{\mu(t)}(\cdot, t)$ has a convergent subsequence in $L^{3/2}(a\leq|x|\leq b)$, for all $0<a<b$. The (common) limit function $f_\infty$ is non-trivial and a non-negative radial critical point of $\Phi$.
    \end{enumerate}
    
\end{theorem}

\begin{remark}\label{remark:non-neg}
    From \eqref{eq:Duhamel}, the reader may readily verify that if $f_0\geq0$, then $f(x,t)\geq 0$ for all $(x,t)\in \R^2\times \R_+$. Moreover, if $f_0\neq0$, then $f(\cdot,t)\neq0$ for all $t>0$.
\end{remark}

\begin{remark}
    The multiplier $\lambda(f(t))$ in \eqref{eq:EL} depends on $t$ insofar as $f$ depends on $t$. Observe however that $\lambda(f)$ does not depend on the spatial variable $x$.
\end{remark}

While Theorem \ref{thm:EL main} establishes subsequential convergence of the rescaled solution $f(t)$ of \eqref{eq:EL} to critical points of the Radon functional $\Phi$ in the radial case under the hypothesis \eqref{eq:threshold}, it is an open question whether $f(t)$ converges to a radial extremiser of $\Phi$. In fact Christ \cite{Christ} posed the question of whether there actually exist non-negative critical points of $\Phi$ that are not extremisers. This question still remains open to best of the author's knowledge. It is therefore natural to pose the following conjecture about the Euler--Lagrange flow:
\begin{question}\label{conj:extremiser}
    Let $0\neq f_0\in L^{3/2}(\R^2)$ be non-negative and radial, and let $f(t)$ be the solution of \eqref{eq:EL} Then there exists a scaling function $\mu(t)$ such that $\mu(t)^{4/3}f(\mu(t)\cdot,t)$ converges to a radial extremiser of $\Phi$.
\end{question}
Both a positive and a negative resolution to Conjecture \ref{conj:extremiser} would resolve the question posed by Christ, at the very least for the case of radial extremisers. If the conjecture is true, then by the same logic of \cite{Carlen-Carrillo-Loss}, the monotonicity of $\Phi$ along the flow \eqref{eq:EL} implies that the extremisers of $\Phi$ are the \emph{only} critical points. If the conjecture is false, then there exist non-extremising critical points of $\Phi$. \\\\
There are partial results regarding the characterisation of critical points of $\Phi$. Christ \cite{Christ} showed the following:
\begin{theorem}[Christ \cite{Christ}]\label{thm:Christ asymptotic}
    Any critical point $f$ of the functional $\Phi$ has an asymptotic expansion of the form 
    \begin{equation}
        f(x)=\sum_{k=0}^\infty g_k\left(x/|x|\right)|x|^{-2-k},
    \end{equation}
    with each $g_k\in C^\infty(\bbS^1)$. In particular, 
    \begin{align}\label{eq:asymptotic}
        f(x)\leq C|x|^{-2},\ \mbox{ as }\ |x|\to \infty,
    \end{align}
      and $|\nabla^kf(x)|\leq C_k|x|^{-2-k}$ as $|x|\to \infty$ for some constants $C,C_k>0$, $k\in\mathbb{N}$.
\end{theorem}

\begin{remark}\label{remark:C infinity}
    It is also true that any critical point $f$ of $\Phi$ is $C^\infty$ (and therefore bounded). To justify this claim, we invoke the following two facts:
    \begin{enumerate}
        \item The functional $\Phi$ has the same critical points as the functional $$f\mapsto\|\cC f\|_3/\|f\|_{3/2},$$ where $\cC$ is the convolution operator with the parabola in $\R^2$,
        \[\cC f(x):=\int_{\R} f(x_1-y,x_2-y^2/2)dy.\]
        This is observed by Christ in the proof of Theorem $1.4$ in \cite{Christ}, due to the norm-equivalence $\|\cC f\|_3=\|\cR f\|_3$.
        \vspace{0.1cm}
        \item All critical points of $\|\cC f\|_3/\|f\|_{3/2}$ have been shown by Christ and Xue \cite{Christ-Xue} to be $C^\infty$.
    \end{enumerate}
    
\end{remark}
We compliment the asymptotic upper bound \eqref{eq:asymptotic} in Theorem \ref{thm:Christ asymptotic} with a matching lower bound, by utilising the analytical information contained in the Euler--Lagrange equation.
\begin{theorem}\label{thm:critical points}
    If $f\in L^{3/2}(\R^2)$ is a non-negative radial critical point of $\Phi$, then there exist constants $0< m\leq M<\infty$ depending on $f$ such that 
    \begin{equation}
        \frac{m}{1+|x|^2}\leq f(x)\leq \frac{M}{1+|x|^2},
    \end{equation}
    for all $x\in \R^2$.
\end{theorem}
The main novelty in the proofs of Theorems \ref{thm:EL main} (particularly the convergence part) and \ref{thm:critical points} is a careful analysis of the operator $A$ restricted to radial functions (see proofs of Proposition \ref{prop:A smoothing} and Lemma \ref{lemma:mass trapping}) and the resulting smoothing estimate:
\begin{theorem}\label{thm:A smoothing}
    For any non-negative and radial $f\in L^{3/2}(\R^2)$ and $0<a<b$ we have $Af \in H^{1/2,3/2}([a,b])$ and
    \begin{align}\label{eq:A Smoothing}
        \|Af\|_{H^{1/2,3/2}([a,b])}\lesssim_{a,b}\|f\|_{3/2}^2,
    \end{align}     
    where $H^{1/2,3/2}([a,b])$ is a Bessel-potential function space (see the Notation section below for the precise definition). That is, the quadratic operator $A=\cR^*[(\cR\cdot)^2]$ is smoothing of order 1/2 on non-negative radial functions, locally on any interval $[a,b]\subseteq(0,\infty)$.
\end{theorem}
This estimate is to be compared to the classical linear fractional-integral identity
\begin{align}\label{eq:fourier slice}
     \mathcal{R}^*\mathcal{R}f=c(-\Delta)^{-1/2}f,
\end{align}
found in \cite{Helgason}, which holds for all (not necessarily radial) smooth functions $f$. The restriction to radial function in \eqref{eq:A Smoothing} is motivated by several identities for the operator $A$ that hold only in the radial case and which facilitate the analysis carried out in Subsection \ref{subsec:sobolev}. We do not address whether the smoothing order of $1/2$ in Theorem \ref{thm:A smoothing} is optimal. We remark however, that a potential improvement in the smoothness of the estimate would not imply any improvements on the conclusions of Theorems \ref{thm:EL main} and \ref{thm:critical points}.\\\\
The proof of Theorem \ref{thm:A smoothing} relies upon the radiality assumption on the underlying input function. It is not clear to the author whether this analysis can be carried out in the non-radial case without adding considerable complexity and new ideas to the argument. An extensive discussion of this matter can be found in Subsection \ref{sec:non-radial}.\\\\
Below we outline the structure of the paper for the reader's convenience.
\begin{itemize}
    \item In Section \ref{section:well-posedness}, we show that equation \eqref{eq:EL} has a global solution in $L^{3/2}(\R^2)$.
    \item In Section \ref{section:monotonicity}, we prove various monotonicity properties of the flow \eqref{eq:EL} for non-negative initial data, the main one being that the functional $\Phi(f(\cdot,t))$ is non-decreasing along \eqref{eq:EL}.
    \item In Section \ref{section:convergence} we prove precompactness of the trajectory of \eqref{eq:EL} for radial initial data and existence of a non-trivial subsequential limit after rescaling. This completes the proofs of Theorems \ref{thm:EL main} and \ref{thm:A smoothing}.
    \item In Section \ref{section:critical points} we prove Theorem \ref{thm:critical points} and Proposition \ref{prop:extremisers}.
\end{itemize}

\acknowledgement{I am grateful to my advisor Anthony Carbery for his active involvement in the conception and development of this work. I also thank Jonathan Hickman and Paolo Ciatti for their suggestions. This research was funded by EPSRC and the Maxwell Institute for Mathematical Sciences.}

\section*{Notation}
\begin{itemize}
    \item We denote the positive half-line $[0,\infty)$ by $\R_+$.
    \item  Let $x,y\in\R_+$. We write: 
    \begin{enumerate}
        \item $x\lesssim y$ iff there exists a constant $C>0$ such that $x\leq Cy$,
        \item  $x\gtrsim y$ iff $y\lesssim x$,
        \item $x\sim y $ iff $x\lesssim y$ and $y\lesssim x$.
    \end{enumerate}

    \item We denote the support of a function $\phi$ by $\mathrm{supp}(\phi)$.
    \item Let $f:\R^2\to \R$ be a function and $\lambda>0$. We denote $f_\lambda:=f(\lambda \cdot)$.
    \item Let $u:\R^n\times \R_+\to \R$ denote a non-negative function of a spatial variable $x\in \R^n$ and a time variable $t\in \R_+$. For each $t\geq 0$, we denote the state $u(\cdot,t):\R^n\to \R$ of $u$ at time $t$ as $u(t)$. For any $p>1$, as a convention we write
    \[\int u(t)^p:=\int_{\R^n} u(x,t)^pdx.\]
    When the specific value of $t$ does not play a role, we will suppress it and write $\int u^p$ instead of $\int u(t)^p$, and any statement for $\int u^p$ is meant for all $t\geq 0$.
    \item For $\lambda>0$ and $u:\R^n\times \R_+\to \R$ we denote $u_\lambda(t):=u(\lambda\cdot,t)$.
    \item For any domain $\Omega\subseteq \R^n$, we denote the space of distributions on $\Omega$ by $\cD(\Omega)$.
    \item We use the standard definition for the Fourier transform in $\R^n$:
    \[\widehat{f}(\xi):=\int_{\R^n} f(x) e^{-2\pi i x\cdot \xi}dx.\]
    \item For $\alpha\neq0$ we use $(-\Delta)^\alpha$ to denote the Fourier multiplier with symbol $(2\pi|\xi|)^{2\alpha}$ in $\R^2$. For $\alpha>0$, we denote the Riesz potential of order $\alpha$ $(-\Delta)^{-\alpha/2}$ by $I_\alpha$.
    \item Let $n\in \Z_+, s>0, p>1$ and let $\Omega$ be a suitable (e.g. bounded Lipschitz) domain in $\R^n$. We define the Bessel-potential space 
    \[H^{s,p}(\R^n):=\{u\in\cD(\R^n):\ \|(I-\Delta)^{s/2}u\|_{L^p(\R^n)}<\infty\}, \]
     with norm 
    \[\|u\|_{H^{s,p}(\R^n)}:=\|(I-\Delta)^{s/2}u\|_{L^p(\R^n)}.\]
    The domain-restricted space $H^{s,p}(\Omega)$ is defined as
    \[H^{s,p}(\Omega):=\{u\in \cD(\Omega): \exists\, U\in H^{s,p}(\R^n):\ U|_\Omega=u\},\]
    with the quotient norm 
    \[\|u\|_{H^{s,p}(\Omega)}:=\inf\{\|U\|_{H^{s,p}(\R^n)}: U\in H^{s,p}(\R^n),\ U|_{\Omega}=u\} \]
    \item If $X$ is a Banach space, $(x_n)$ is a sequence in $X$ and $x\in X$, we denote weak convergence of  $(x_n)$ to $x$ by $x_n\rightharpoonup x$.
\end{itemize}

\section{Well-Posedness of the Euler--Lagrange Flow}\label{section:well-posedness}
We first show that the operator $A$ is well-defined for $f\in L^{3/2}(\R^2)$ and that $A: L^{3/2}(\R^2)\to L^{3}(\R^2)$. Indeed, by the boundedness of $\cR^*,\cR$, we have
\begin{align}\label{eq:A bound}
    \left(\int |Af|^3\right)^{1/3}&=\left(\int \big|\cR^*\left(|\cR f|\,\cR f\right)\big|^3\right)^{1/3}\nonumber\\
    &\leq\|\cR\|\|(\cR f)^2\|_{3/2}=\|\cR\|\|\cR f\|_3^2\leq\|\cR\|^3\|f\|_{3/2}^2.
\end{align}
Furthermore, if $0\neq f\in L^{3/2}(\R^2)$, then $\lambda(f)$ is also well-defined. Indeed, by H\"older's inequality we have
\begin{align}\label{eq:lambda upper bound}
    \int|f|^{1/2}|A f|^2\leq\|f\|_{3/2}^{1/2}\|Af\|_3^2\leq \|\cR\|^6\|f\|_{3/2}^{4+\frac{1}{2}}<\infty.
\end{align} 
We now consider the issue of existence and uniqueness of solutions to \eqref{eq:EL}. The presence of the square-root factor in $\int |f|^{1/2}|Af|^2\mathrm{sgn}(f)$ poses a mild obstacle in applying the classical Cauchy--Lipschitz--Picard--Lindel\"of theory, which requires the infinitesimal generator of the flow to be locally Lipschitz in the state-space $L^{3/2}(\R^2)$. We get around this issue by introducing an appropriate change of variables. In the rest of this section, let $f(t)$ be the formal solution of \eqref{eq:EL} with $f_0\geq 0$. By Remark \ref{remark:non-neg}, we have $f(t)\geq 0$ for all $t\geq 0$. We define 
\begin{gather*}
    \Lambda(t):=\int_0^t\lambda(f(s))ds, \qquad
    \psi( t):=e^{\Lambda(t)}f(t).
\end{gather*}
The dependence of $\Lambda(t)$ on $f(t)$ is implicit and will always be clear from context. We will abuse notation and write $\Lambda(f)$ instead of $\Lambda(t)$ when the particular value of $t$ is irrelevant. The function $\psi$ satisfies 
\begin{align*}
    \partial_t \psi&= e^{\Lambda(f)}\partial_t f+\lambda(f)e^{\Lambda(f)}f\\
    &=e^{\Lambda(f)}A(f)^2-\lambda(f)e^{\Lambda(f)}{f}+\lambda(f)e^{\Lambda(f)}{ f}\\
    &=e^{\Lambda(f)}A(f)^2\\ 
    &=e^{\Lambda(f)}A(e^{-\Lambda(f)}\psi)^2\\
    &=e^{-3\Lambda(f)}A(\psi)^2.
\end{align*}
We may write 
\begin{gather*}
    \int \psi^{3/2}=e^{3\Lambda(f)/2}\int f^{3/2},
\end{gather*}
so that 
\begin{gather*}
    e^{-3\Lambda(f)}=\left(\frac{\int f^{3/2}}{\int \psi^{3/2}}\right)^2.
\end{gather*}
We prove the following fundamental invariance property of \eqref{eq:EL}.
\begin{lemma}\label{le:norm invariance}
    Let $f(t)$ be the solution of \eqref{eq:EL} with non-negative initial data $f_0\in L^{3/2}(\R^2)$. Then $\int f(t)^{3/2}$ is constant in $t$.
\end{lemma}
\begin{proof}
    By \eqref{eq:EL} and the definition \eqref{eq:lambda} of $\lambda$  we have
    \begin{align}\label{eq:omega on the sphere}
        \frac{d}{dt}\int f(t)^{3/2}&=\frac{3}{2}\int f(t)^{1/2}\partial_t f(t)=\frac{3}{2}\int f(t)^{1/2}\left[\left(Af(t)\right)^2-\lambda(f(t))f(t)\right]\\
        &=\frac{3}{2}\left[\lambda(f(t))\int f(t)^{3/2}-\lambda(f(t))\int f(t)^{3/2}\right]=0 \nonumber
    \end{align}
\end{proof}
By Lemma \ref{le:norm invariance}, $\int f(t)^{3/2}=\int f_0^{3/2}$ for all $t\geq 0$ and the equation for $\psi$ becomes
\begin{gather}\label{eq:modEL}
    \partial_t \psi= \left(\frac{\int  f_0^{3/2}}{\int \psi^{3/2}}\right)^2 A(\psi)^2,\  \ \psi(x,0)=e^{\Lambda(f_0)} f_0(x)= f_0(x).
\end{gather}
We show that equation \eqref{eq:modEL} is well-defined. Indeed, we have
\begin{gather}\label{eq:lower bound}
    \int \psi(t)^{3/2}= e^{3\Lambda(t)/2}\int f(t)^{3/2}\geq \int  f_0^{3/2}>0,\ t\geq 0,
\end{gather}
since $\Lambda(t)\geq0$. We may express the solution $f(t)$ to \eqref{eq:EL} as a function of the solution to \eqref{eq:modEL} as follows:
\[ f(t) := \frac{\psi(t)}{\|\psi(t)\|_{3/2}}\| f_0\|_{3/2}.\]
Well-posedness of \eqref{eq:modEL} is equivalent to well-posedness of the following:
\begin{equation}\label{eq:main_ode}
\frac{d}{dt} \psi(t) = c{\Psi}(\psi), \;\;\;\ \psi(0) = \psi_0,
\end{equation}
where $c:=(\int f_0^{3/2})^2$ and the nonlinear map ${\Psi}: L^{3/2} \to L^{3/2}$ is given by\footnote{Nothing substantial changes if we replace $|\psi|$ in the numerator here by $\psi$.}  
\[ {\Psi}(\psi) = \left[\frac {\cR^\ast (\cR|\psi|)^2}{\int |\psi|^{3/2}}\right]^2, \quad \psi\neq 0,\]
and $\Psi(0)=0$. That ${\Psi}: L^{3/2} \to L^{3/2}$ is easily seen by examining the chain
\[ \psi \mapsto |\psi| \mapsto \cR|\psi| \mapsto (\cR|\psi|)^2 \mapsto \cR^\ast(\cR|\psi|)^2 \mapsto (\cR^\ast(\cR|\psi|)^2)^2\]
\[ L^{3/2} \to L^{3/2} \to L^3 \to L^{3/2} \to L^3 \to L^{3/2}.\]
Below we will show that for all $\psi \in L^{3/2}$ we have, with $\|\cR\|$ denoting the norm of $\cR \in \mathcal{L}(L^{3/2}, L^3)$,
\begin{equation}\label{eq:normG}
\|{\Psi}(\psi)\|_{3/2} \leq \|\cR\|^6 \|\psi\|_{3/2}.
\end{equation}

\medskip
\noindent
We wish to apply the classical Cauchy--Lipschitz--Picard--Lindel\"of theorem in order to be able to assert that, for arbitrary initial data $\psi_0 \in L^{3/2}$, the initial value problem \eqref{eq:main_ode} has a unique solution in $C^1([0, \infty), L^{3/2})$.

\medskip
\noindent
\begin{theorem}[Br\'ezis \cite{Brezis}, Theorem 7.3]\label{thm:Brezis}
 Let $X$ be a Banach space and suppose $\Psi : X \to X$. Suppose that $\Psi$ is globally Lipschitz, that is, there is a constant $K < \infty$ such that for all $x, y \in X$ we have 
 \[ \|\Psi(x) - \Psi(y)\|_X \leq K \|x - y\|_X.\] 
 Then for all $x_0 \in X$ 
 there exists a unique solution $x(t) \in C^1([0, \infty), X)$ 
 to the problem
 \[ \frac{d}{dt} x(t) = \Psi(x), \;\;\;\ x(0) = x_0.\]
 \end{theorem}

It may also easily be shown that if $\Psi (x_0) \neq 0$, then the solution $x(t)$ satisfies $\Psi (x(t)) \neq 0$ for all $t \geq 0$. In fact, in our application, if $\psi_0 \geq 0$, the positivity of $\Psi$ ensures that not only is $\psi(t) \neq 0$ for all $t\geq 0$, but indeed that $\psi(t)$ is pointwise increasing as time increases.
 
\medskip
\noindent
In order to apply this result, we need:

\begin{proposition}\label{prop:Psi Lip}
For the function $\Psi$ as above, there is a constant $K$ such that for all $u,v \in L^{3/2}$,
\[ \| \Psi(u) - \Psi(v) \|_{3/2} \leq  K\|\cR\|^6 \|u-v\|_{3/2}.\]
\end{proposition}

\begin{proof}
We note that it suffices to prove this for $u$ and $v$ nonnegative, since if we have that, then for general $u$, $v$ we have
\[ \| \Psi(u) - \Psi(v) \|_{3/2} =  \| \Psi(|u|) - \Psi(|v|) \|_{3/2} \leq K \|\cR\|^6 \||u|-|v|\|_{3/2} \leq  K\|\cR\|^6 \|u-v\|_{3/2}.\]

\medskip
\noindent
We thus assume $u$, $v\geq0$, and 
we proceed as above by decomposing $\Psi$ into simpler pieces. 

\medskip
\noindent
Let $S:L^3 \to L^{3/2}$ be the squaring map $h \mapsto h^2$. For $u \in L^{3/2}_+$ fixed let $M_u: L^{3} \to L^{3}$ be the map $h \mapsto h/\int u^{3/2}$.

\medskip
We record some basic facts about these maps which we shall use below:

\begin{itemize}
\item $\|\cR u - \cR v\|_{3} \leq \|\cR\| \|u-v\|_{3/2}$
\item $\|\cR^\ast u - \cR^\ast v\|_{3} \leq \|\cR\| \|u-v\|_{3/2}$
\item $\|Sh\|_{3/2} = \|h\|_3^2$
\item $\|Sh - Sk\|_{3/2} \leq (\|h\|_{3} + \|k \|_{3})\|h-k\|_{3}$
\item For $u,v \in L^{3/2}_+$, 
\[\|M_u h - M_v k\|_3 \leq \frac{1}{\int u^{3/2}}\|h-k\|_3 + \frac{3}{2}\frac{\|k\|_3}{\int u^{3/2}\int v^{3/2}} (\|u - v\|_{3/2})( \|u\|_{3/2}^{1/2} +  \|v\|_{3/2}^{1/2}).\] 
\end{itemize}
Only the last of these needs discussion: we have

\[ \|M_u h - M_v k\|_3 = \left\| \frac{h}{\int u^{3/2}} - \frac{k}{\int v^{3/2}} \right\|_3 = \left\| \frac{h}{\int u^{3/2}} - \frac{k}{\int u^{3/2}} + \frac{k}{\int u^{3/2}} - \frac{k}{\int v^{3/2}} \right\|_3\]
\[ \leq  \left\| \frac{h}{\int u^{3/2}} - \frac{k}{\int u^{3/2}}\right\|_3 + \left\|  \frac{k}{\int u^{3/2}} - \frac{k}{\int v^{3/2}} \right\|_3\]
\[ \leq \frac{1}{\int u^{3/2}}\|h-k\|_3 + \|k\|_3 \left|\frac{\int u^{3/2} - \int v^{3/2}}{\int u^{3/2}\int v^{3/2}}\right|,\]
and
\begin{align*}
    \left|\int u^{3/2} - v^{3/2}\right| = \left|\|u\|_{3/2}^{3/2} - \|v\|_{3/2}^{3/2}\right| &\leq \frac{3}{2} (\|u\|_{3/2} - \|v\|_{3/2})\left( \|u\|_{3/2}^{1/2} +  \|v\|_{3/2}^{1/2}\right)\\ &\leq  \frac{3}{2} (\|u - v\|_{3/2})\left( \|u\|_{3/2}^{1/2} +  \|v\|_{3/2}^{1/2}\right),
\end{align*} 
thus establishing the last bullet point.

\medskip
\noindent
Note that for $u \in L^{3/2}_+$, we have
\[ \Psi(u) = S \circ M_u \circ \cR^\ast \circ S \circ \cR u.\]

With this in mind, we now let, for $u \in L^{3/2}_+$,
\[\Psi_1 (u) = M_u \circ \cR^\ast \circ S \circ \cR u, \;\;
\Psi_2 (u) = \cR^\ast \circ S \circ \cR u \mbox{ and } \Psi_3 (u) =  S \circ \cR u.\]

We immediately note that 
\begin{align*}
    \|\Psi(u)\|_{3/2} &= \|S(\Psi_1(u))\|_{3/2} = \|\Psi_1(u)\|_3^2 = \|M_u(\Psi_2(u))\|_3^2 =\frac{\|\Psi_2(u)\|_{3}^2}{(\int u^{3/2})^2}
    \\ &= \|u\|_{3/2}^{-3} \|\|\cR^\ast (\Psi_3(u))\|_3^2 \leq \|\cR\|^2 \|u\|_{3/2}^{-3} \|\Psi_3(u)\|_{3/2}^2 \\
    &= \|\cR\|^2 \|u\|_{3/2}^{-3} \| S(\cR u)\|_{3/2}^2 
    \\&= \|\cR\|^2 \|u\|_{3/2}^{-3} \| \cR u\|_3^4
\leq \|\cR\|^6 \|u\|_{3/2}^{-3}  \|u\|_{3/2}^{4} = \|\cR\|^6 \|u\|_{3/2},
\end{align*}

establishing \eqref{eq:normG}.

\medskip
\noindent
Now we consider the Lipschitz condition as in the statement of the proposition. Note first that if $u$ and $v$ are such that $\|u-v\|_{3/2} \geq  \max\{\|u\|_{3/2}, \|v\|_{3/2}\}/2$, we have
\begin{align*}
    \|\Psi(u) - \Psi(v)\|_{3/2} &\leq \|\Psi(u)\|_{3/2} + \|\Psi(v)\|_{3/2}\\
    &\leq 2\|\cR\|^6 \max\{\|u\|_{3/2}, \|v\|_{3/2}\} \leq 4\|\cR\|^6 \|u-v\|_{3/2},
\end{align*} 
by \eqref{eq:normG}. Hence, in proceeding, we may assume that $\|u\|_{3/2} \leq \|v\|_{3/2}$ and that $\|u-v\|_{3/2} < \|v\|_{3/2}/2$.

\medskip
\noindent
For such $u, v \in L^{3/2}_+$, using the above bullet points, we have 
\begin{align*}
     \| \Psi(u) - \Psi(v) \|_{3/2} &= \|S(\Psi_1(u)) - S(\Psi_1(v)) \|_{3/2}\\
     &\leq  (\|\Psi_1(u)\|_{3} + \|\Psi_1(v) \|_{3})\|\Psi_1(u)-\Psi_1(v)\|_{3}\\ 
    &= (\|\Psi_1(u)\|_{3} + \|\Psi_1(v) \|_{3})\|M_u(\Psi_2(u))-M_v(\Psi_2(v))\|_{3}\\
    &\leq (\|\Psi_1(u)\|_{3} + \|\Psi_1(v) \|_{3})
    \Bigg(\frac{1}{\int u^{3/2}}\|\Psi_2(u)-\Psi_2(v)\|_3 \\
    &\qquad+ \frac{2\|\Psi_2(v)\|_3}{\int u^{3/2}\int v^{3/2}} (\|u - v\|_{3/2})( \|u\|_{3/2}^{1/2} +  \|v\|_{3/2}^{1/2})\Bigg).
\end{align*}

Now 
\begin{align*}
    \|\Psi_2(u)-\Psi_2(v)\|_3 &= \|\cR^\ast \Psi_3(u) - \cR^\ast \Psi_3(v)\|_3 \leq \|\cR\| \|\Psi_3(u) - \Psi_3(v)\|_{3/2}\\
    &= \|\cR\| \|S(\cR u) - S(\cR v)\|_{3/2} \\
    &\leq \|\cR\|(\|\cR u\|_{3} + \|\cR v\|_{3})\|\cR u-\cR v\|_{3}\\ &\leq \|\cR\|^3 (\|u\|_{3/2} + \|v\|_{3/2})(\|u-v\|_{3/2}).
\end{align*}

\medskip
Set $A= \|v\|_{3/2}$. Thus, using $\|u\|_{3/2} \geq A/2$ and $\|u\|_{3/2} \leq 3A/2$, we obtain 
\[\| \Psi(u) - \Psi(v) \|_{3/2} \leq C\|\Psi_1(v)\|_3 \left( A^{-3/2} A \|\cR\|^3  + \|\Psi_2(v)\|_3 A^{-3}A^{1/2}\right) \|u-v\|_{3/2}.\]
And we have
\[ \|\Psi_2(v)\|_3 \leq C \|\cR\|^3 A^2, \]
and
\[\|\Psi_1(v)\|_3 = \|M_{|v|}(\Psi_2(v))\|_{3} \leq A^{-3/2} \| \Psi_2(v)\|_3 \leq A^{1/2}\|\cR\|^3, \] 
so that finally, for such $u,v$,
\[ \| \Psi(u) - \Psi(v) \|_{3/2} \leq C A^{1/2}\|\cR\|^3A^{-1/2} \|\cR\|^3 \|u-v\|_{3/2} = C \|\cR\|^6 \|u-v\|_{3/2}.\]
\end{proof}
Thus the problem \eqref{eq:main_ode} indeed has a unique solution 
for all times $t \geq 0$, and moreover the solution increases pointwise if $\psi_0 \geq 0$.
\medskip

Since equation \eqref{eq:EL} is autonomous, it satisfies the semigroup property. Indeed, 
let $\mathcal{F}(f)=Af^2-\lambda(f) f$ and let $T(t)f_0$ denote the solution of $\frac{df}{dt}=\mathcal{F}(f)$ at time $t$ with initial data $f_0$. We know that $T(t)f_0$ is well-defined for all $t\geq 0$. We will show that for every $s,t\geq 0$, we have
$$T(t+s)f_0=T(t)T(s)f_0.$$
For  $\tau\leq t$, let $u(\tau)=T(\tau)f_0$, $w(\tau)=u(\tau+s)$ and $v(\tau)=T(\tau)u(s)$. Since $u,v$ are solutions to $\frac{df}{dt}=\mathcal{F}(f)$, we have
\begin{gather*}
    \frac{d w}{d\tau}(\tau)=\frac{d u}{d\tau}(\tau+s)=\mathcal{F}(u(\tau +s))=\mathcal{F}(w(\tau)),
\end{gather*}
and 
\begin{gather*}
    \frac{d v}{d\tau}(\tau)=\mathcal{F}(v(\tau)).
\end{gather*}
Moreover, we have $w(0)=u(s)=v(0)$. By uniqueness, $w\equiv v$.

\section{Monotonicity Properties of the Radon Flow}\label{section:monotonicity}
\subsection{Monotonicity Properties}
In this section, $f\in C^1\left(\R_+, L^{3/2}(\R^2)\right)$ will denote the solution of equation \eqref{eq:EL} with non-negative initial data $f_0\in L^{3/2}(\R^2)$, the existence of which we have established in Section \ref{section:well-posedness}. Moreover, we have established that $f(x,t)\geq 0$ for all $(x,t)\in\R^2\times \R_+$. The following is the key monotonicity property of \eqref{eq:EL}.
\begin{proposition}\label{prop:monotone}
   We have $$\frac{d}{dt}\Phi (f(t))^3\geq 0, \mbox{ for all } t>0,$$
\end{proposition}
\begin{proof}  
Since $\|f(t)\|_{L^{3/2}(\R^2)}$ is constant by Lemma \ref{le:norm invariance}, we have
\begin{align}\label{eq:Phi derivative}
    \frac{d}{dt}\Phi(f(t))^3&=\frac{d}{dt}\frac{\|\cR f(t)\|_3^3}{\|f(t)\|_{3/2}^3}=3\|f_0\|_{3/2}^{-3}\int(\cR f(t))^2\partial_t \cR f(t) \\
    &=3\|f_0\|_{3/2}^{-3}\int (\cR f)^2 \cR(\partial_t f)=3\|f_0\|_{3/2}^{-3}\int \cR^*[(\cR f)^2]\partial_t f\nonumber\\ &=3\|f_0\|_{3/2}^{-3}\int A(f)\partial_t f \nonumber,
\end{align}
By \eqref{eq:EL} and since $f(t), Af(t)$ and $\lambda(f(t))$ are non-negative, we have
\begin{gather*}
     \partial_tf(t)\leq0\iff A(f(t))\leq\lambda(f(t))^{1/2}f(t)^{1/2}.
\end{gather*}
In other words, $\partial_t f(t)$ and $Af(t)-\lambda(f(t))^{1/2}f(t)^{1/2}$ have the same sign pointwise. Hence, 
\begin{gather*}
    \partial_t f(t) A(f(t))\geq \partial_t f(t) \lambda(f(t))^{1/2}f(t)^{1/2},
\end{gather*}
and therefore
\begin{align*}
    \frac{d}{dt}\Phi(f(t))^3&\geq 3\|f_0\|_{3/2}^{-3}\int \lambda(f)^{1/2}f^{1/2} \partial_t f\\
    &=2\|f_0\|_{3/2}^{-3}\lambda(f)^{1/2}\frac{3}{2}\int f^{1/2}\partial_t f=2\|f_0\|_{3/2}^{-3}\lambda(f)^{1/2}\frac{d}{dt}\|f(t)\|_{3/2}^{3/2}=0.
\end{align*}
\end{proof}
Proposition \ref{prop:monotone} immediately yields the following.
\begin{corollary}
    $\Phi(f(t))$ is non-decreasing.
\end{corollary}
Moreover, the following is true:

\begin{proposition}\label{prop:lambda monotone}
    The function $\lambda(f(t))$ is differentiable and
    \begin{gather*}
        \frac{d}{dt}\lambda(f(t))\geq 0\ \mbox{ for all }\ t\geq 0.
    \end{gather*}
    In particular, we have
    \begin{gather*}
        \lambda(f(t))\geq\lambda(f_0)=\frac{\int f_0^{1/2}A(f_0)^2}{\int f_0^{3/2}}.
    \end{gather*}
\end{proposition}

To prove Proposition \ref{prop:lambda monotone}, we first need the following lemma.

\begin{lemma}\label{lemma:Cauchy-Schwarz}
    Define the symmetric bilinear form
    \begin{gather*}
        B(f,g):= \cR^*\left[\cR f\cdot \cR g\right].
    \end{gather*}
    Then $B$ satisfies the Cauchy--Schwarz inequality:
    \begin{equation*}
        |B(f,g)(x)|^2\leq B(f,f)(x)B(g,g)(x)=Af(x)Ag(x),
    \end{equation*}
    for non-negative $f,g$. Moreover, for every $f\geq0$, the linear operator $B(f,\cdot)$ is self-adjoint and satisfies the pointwise Cauchy--Schwarz inequality:
    \begin{gather*}
        B(f,gh)(x)^2\leq B(f,g^2)(x)B(f,h^2)(x),
    \end{gather*}
    for all $g,h\geq 0$ and almost every $x$.
\end{lemma}
\begin{proof}
    Let $f,g$ be non-negative functions and let $x\in\R^2$. By the triangle inequality for integrals and the standard Cauchy--Schwarz inequality, we have
    \begin{align*}
        |B(f,g)(x)|^2&\leq\left(\int_{\bbS^1} |\cR f(\omega\cdot x) \cR g(\omega\cdot x) |d\omega\right)^2\\
        &\leq\int_{\bbS^1}(\cR f)(\omega\cdot x)^2 d\omega\int_{\bbS^1}(\cR g)(\omega\cdot x)^2 d\omega\\
        &=\cR^*\left[(\cR f)^2\right](x)\cR^*\left[(\cR g)^2\right](x)=Af(x)Ag(x).
    \end{align*}
    Fix $f$. We have
    \begin{align*}
        \int g B(f,h)&=\int g \cR^*\left[\cR f\cdot \cR h\right]= \int \cR g \cdot \cR f \cdot \cR h \\
        &=\int \cR^*\left[\cR f\cdot \cR g\right] h=\int B(f,g) h.
    \end{align*}
    Furthermore, for fixed $x\in\R^2$ and any $g,h\geq0$ we have by standard Cauchy--Schwarz
    \begin{align*}
        B(f,gh)(x)^2&=\left[\cR^*[\cR f\cdot \cR(gh)](x)\right]^2\leq \left[\cR^*[\cR f\cdot(\cR\{g^2\})^{1/2}(\cR\{h^2\})^{1/2}](x)\right]^2\\
        &\leq \cR^*[\cR f \cdot \cR(g^2)](x)\cR^*[\cR f\cdot \cR(h^2)](x)=B(f,g^2)(x)B(f,h^2)(x).
    \end{align*}
\end{proof}

\begin{proof}[Proof of Proposition \ref{prop:lambda monotone}]
    Set $$J(t):=\int f(t)^{1/2}A(f(t))^2.$$
    Since $\int f(t)^{3/2}$ is constant by Lemma \ref{le:norm invariance}, we need only show that $J(t)$ is differentiable and
    $$\frac{d}{dt}J(t)\geq 0.$$
    Since $Af(t)=\cR^*\left([\cR f(t)]^2\right),$ it is evident that $Af(t)(x)$ is differentiable in $t$ for almost every $x\in\R^2$ and 
    \begin{align}\label{eq:A derivative}
    \partial_t A(f(t))&=\partial_t \cR^*[(\cR f(t))^2]=\cR^*[\partial_t(\cR f(t))^2]\nonumber\\
    &=2\cR^*[\cR f(t)\cdot\partial_t \cR f(t)]=2\cR^*[\cR f(t)\cdot \cR(\partial_tf(t))].
    \end{align}
    Hence, $J$ is differentiable and we have
    \begin{gather*}
        \frac{d}{dt} J(t) = \int\frac{1}{2}f(t)^{-1/2}\partial_t f(t)A(f(t))^2+\int f(t)^{1/2}\partial_t\left[A(f(t))^2\right]
    \end{gather*}
By \eqref{eq:EL}, we have
\begin{gather*}
     \partial_t f(t)\leq 0 \iff A(f(t))^2\leq\lambda(f(t))f(t),\\
     \partial_t f(t)\geq 0 \iff A(f(t))^2\geq\lambda(f(t))f(t),
\end{gather*}
so that in both cases we get
\begin{gather*}
    \partial_t f(t) A(f(t))^2\geq \partial_t f(t) \lambda(f(t))f(t).
\end{gather*}
For the first integral, we have
\begin{align*}
     \int\frac{1}{2}f(t)^{-1/2}\partial_t f(t)A(f(t))^2
     &\geq \frac{1}{2}\int\lambda(f(t))f(t)^{1/2}\partial_t f(t)\\
     &=\frac{\lambda(f(t))}{2}\int f(t)^{1/2}\partial_t f(t)\\
     &=\frac{\lambda(f(t))}{3}\frac{d}{dt}\int f(t)^{3/2}=0.
\end{align*}
For the second integral, we have
\begin{gather*}
    \int f(t)^{1/2}\partial_t\left[A(f(t))^2\right]=2\int f(t)^{1/2} A(f(t))\partial_t A(f(t)),
\end{gather*}
and thus by \eqref{eq:A derivative} and the self-adjoint property of $B$ we have
\begin{align*}
    \int f(t)^{1/2}\partial_t\left[A(f(t))^2\right]&= 4\int f(t)^{1/2} A(f(t)) \cR^*[\cR f(t)\cdot \cR(\partial_t f(t))]\\
    &=4\int f(t)^{1/2}A(f(t)) B(f(t),\partial_t f(t))\\
    &=4\int f(t)^{1/2}A(f(t))B(f(t),A(f(t))^2)\\
    &\qquad-4\lambda(f(t))\int f(t)^{1/2}A(f(t)) B(f(t),f(t))\\
    &=4\int B(f(t),f(t)^{1/2}A(f(t)))A(f(t))^2\\
    &\qquad-4\lambda(f(t))\int f(t)^{1/2}A(f(t))^2\\
    &=4\left[\int B(f(t),f(t)^{1/2}A(f(t)))A(f(t))^2-\frac{J(t)^2}{\int f(t)^{3/2}}\right].
\end{align*}
By the self-adjoint property of $B(f,\cdot)$ and Cauchy--Schwarz, we have 
\begin{align*}
    J(t)=\int f(t)^{1/2}Af(t)^2&=\int f(t)^{1/2}Af(t) B(f(t),f(t))\\
    &=\int f(t) B(f(t),f(t)^{1/2}Af(t))\\
    &\leq \left(\int f(t)^{3/2}\right)^{1/2}\left(\int f(t)^{1/2} B(f(t),f(t)^{1/2}Af(t))^2\right)^{1/2},
\end{align*}
and by Cauchy--Schwarz for $B(f(t),\cdot)$ and the nonnegativity of $f$, we have
\begin{align*}
    B(f(t),f(t)^{1/2}Af(t))(x)^2&\leq B(f(t),f(t))(x)B(f(t),Af(t)^2)(x)\\
    &=Af(t)(x)B(f(t),Af(t)^2)(x).
\end{align*}
Combining the last two inequalities, and using self-adjointness of $B$ again, we get
\begin{align*}
    \frac{J(t)^2}{\int f(t)^{3/2}}&\leq\int f(t)^{1/2}Af(t)B\left( f(t),Af(t)^2\right)\\
    &=\int B(f(t),f(t)^{1/2}Af(t))Af(t)^2, 
\end{align*}
and thus $$\int f(t)^{1/2}\partial_t\left[A(f(t))^2\right]\geq 0.$$
\end{proof}

\section{Convergence to Critical Points}\label{section:convergence}
In this section, we prove Theorem \ref{thm:A smoothing} and the remaining conclusions of Theorem \ref{thm:EL main}. We do so in three steps outlined below. In what follows, $f$ will be the solution of \eqref{eq:EL} with non-negative initial data $f_0\in L^{3/2}(\R^2)$, unless stated otherwise.
\begin{itemize}
    \item In Subsection \ref{subsec:dissipation}, we prove a dissipation property of \eqref{eq:EL}, which forces the solution of \eqref{eq:EL} to approach (possibly trivial at this point) solutions of the Euler--Lagrange equation.
    \item In Subsection \ref{subsec:sobolev}, we show that if the initial data of \eqref{eq:EL} is non-negative and radial, then the solution $f(t)$ of \eqref{eq:EL} converges subsequentially to a (possibly $0$ at this point) limit function $f_\infty$ in $L^{3/2}(a\leq|x|\leq b)$ for every $0<a<b$.
    \item In Subsection \ref{subsec:convergence}, we show that there exists a suitable rescaling of the solution of \eqref{eq:EL} such that the limit obtained in Subsection \ref{subsec:sobolev} is non-trivial.
\end{itemize}

\subsection{Dissipation Estimates}\label{subsec:dissipation}

\begin{proposition}\label{prop:extr_integrability f}
    We have
    $$\int_0^\infty\|\partial_t  f(t)\|_{3/2}^3 \, dt= \int_0^\infty \|Af(t)^2-\lambda(f(t))f(t)\|_{3/2}^3 \, dt<\infty$$ 
\end{proposition}
\begin{proof}
    For simplicity, we may assume that $\|f_0\|_{3/2}=1$. From \eqref{eq:Phi derivative},\eqref{eq:EL} and since by Lemma \ref{le:norm invariance} we have
    \[\int f^{1/2}\partial_t f=\frac{2}{3}\frac{d}{dt}\|f(t)\|_{3/2}^{3/2}=0,\]
    we get the identity
    \begin{align*}
        \frac{1}{3}\frac{d}{dt}\Phi(f(t))^3&=\int Af \partial_t f-\lambda(f)^{1/2}\int f^{1/2}\partial_t f\\
        &=\int\left( Af-\lambda(f)^{1/2}f^{1/2}\right)\partial_t f\\
        &=\int\left( Af-\lambda(f)^{1/2}f^{1/2}\right)\left(Af^2-\lambda(f)f \right).
    \end{align*}
    We have
    \begin{gather*}
        \int \left| Af-\lambda(f)^{1/2}f^{1/2}\right|^3= \int \left| Af-\lambda(f)^{1/2}f^{1/2}\right| \cdot \left| Af-\lambda(f)^{1/2}f^{1/2}\right|^2,
    \end{gather*}
    and by applying the elementary inequality $|\sqrt{x}-\sqrt{y}|\leq\sqrt{|x-y|}$ (valid for $x, y \geq 0$), the above is
    \begin{gather*}
        \leq\int \left| Af-\lambda(f)^{1/2}f^{1/2}\right| \cdot \left| Af^2-\lambda(f)f\right|.
    \end{gather*}
    Since $Af,f\geq 0$, the expressions $Af^2-\lambda(f)f$ and $Af-\lambda(f)^{1/2}f^{1/2}$ have the same sign almost everywhere. Thus,
    \begin{gather*}
        \int \left| Af-\lambda(f)^{1/2}f^{1/2}\right|^3\leq\int\left( Af-\lambda(f)^{1/2}f^{1/2}\right)\left(Af^2-\lambda(f)f \right)=\frac{1}{3}\frac{d}{dt}\Phi(f(t))^3.
    \end{gather*}
    Finally, we have
    \begin{align*}
        \left\|Af^2-\lambda(f)f\right\|_{3/2}^3&=\left(\int \left|Af^2-\lambda(f)f\right|^{3/2}\right)^2\\
        &=\left(\int \left|Af+\lambda(f)^{1/2}f^{1/2}\right|^{3/2}\left|Af-\lambda(f)^{1/2}f^{1/2}\right|^{3/2}\right)^2\\
        &\leq\int \left|Af+\lambda(f)^{1/2}f^{1/2}\right|^{3}\int \left|Af-\lambda(f)^{1/2}f^{1/2}\right|^{3}\\
        &\leq 4\left(\|Af\|_3^3+\lambda(f)^{3/2}\|f\|_{3/2}^{3/2}\right) \int \left|Af-\lambda(f)^{1/2}f^{1/2}\right|^{3}.        
    \end{align*}
    It remains to show that the factor multiplying the integral in the last expression is bounded. This may be readily verified by \eqref{eq:normG} and \eqref{eq:lambda upper bound}.
\end{proof}

We now show that $\|\partial_t f(t)\|_{3/2}$ vanishes as $t \to \infty$.
\begin{proposition}\label{prop:extr_conv}
    We have
    \[\|Af(t)^2-\lambda(f(t))f(t)\|_{3/2}=\|\partial_t f(t)\|_{3/2}\to 0,\ \mbox{ as }\ t\to \infty.\]
\end{proposition}
To prove this, we need the following simple variant of Barbalat's lemma, the proof of which may be found in Appendix \ref{ap:barbalat}.
\begin{lemma}\label{le:barbalat}
    Let $u\in L^1(\R_+)$ be differentiable and suppose $u'=\xi+\nu$, where $\xi\in L^1(\R_+)$ and $\nu\in L^\infty(\R_+)$. Then $u$ is uniformly continuous and 
    \[\lim_{t\to +\infty}u(t)=0.\]
\end{lemma}
\begin{proof}[Proof of Proposition \ref{prop:extr_conv}]
        Let $h(t):=\|\partial_t f(t)\|_{3/2}^{3/2}$. We know from Proposition \ref{prop:extr_integrability f} that $h^2 \in L^1(\R_+)$. Moreover, we have
        \begin{gather*}
            h(t)=\|Af(t)^2-\lambda(f(t))f(t)\|_{3/2}^{3/2}\lesssim\|Af(t)\|_3^{3}+\lambda(f(t))^{3/2}\|f(t)\|_{3/2}^{3/2}\lesssim\|\cR\|^9\|f_0\|_{3/2}^6.
        \end{gather*}
        Thus, $h^2\in L^1(\R_+)\cap L^\infty(\R_+)$. In view of Lemma \ref{le:barbalat}, we need only show that $h^2$ is differentiable and $(h^2)'=\xi+\nu$ for some $\xi\in L^1(\R_+)$ and $\nu\in L^\infty(\R_+)$. By Proposition \ref{prop:lambda monotone}, $\lambda(t)$ and $Af(t)$ are differentiable in $t$. Therefore, $\partial_tf(t)$ is also differentiable in $t$ and we have
        \begin{equation}\label{eq:double derivative}
            \partial_t^2 f(t)=\partial_t[Af(t)^2]-\lambda'(t)f(t)-\lambda(t)\partial_t f(t).
        \end{equation}
        Since the function $s\mapsto |s|^p$ is $C^1$ on $\R$ for $p>1$, we have that $h$ is differentiable and
        \begin{gather*}
            h'(t)=\frac{d}{dt}\int |\partial_tf(t)|^{3/2}=\frac{3}{2}\int |\partial_tf(t)|^{1/2}{\rm sgn}(\partial_t f(t))\partial_t^2f(t).
        \end{gather*}  
        Since $(h^2)'(t)=2h(t)h'(t)$ and $h$ is bounded, it suffices to show that $h'=\xi+\nu$ for some $\xi\in L^1(\R_+)$ and $\nu\in L^\infty(\R_+)$.
        We substitute \eqref{eq:double derivative} in the expression for $h'$ to get
        \begin{align*}
            h'(t)=&3\int_{\R^2}|\partial_t f(t)|^{1/2}{\rm sgn}(\partial_t f(t))Af(t)\partial_tAf(t)\\
            -&\frac{3}{2}\lambda'(t)\int_{\R^2}|\partial_t f(t)|^{1/2}{\rm sgn}(\partial_t f(t))f(t)\\
            -&\frac{3}{2}\lambda(t)\int_{\R^2} |\partial_t f(t)|^{3/2}
        \end{align*}
        We now set
        \begin{equation}\label{eq:xi}
            \xi(t):=-\frac{3}{2}\lambda'(t)\int_{\R^2}|\partial_t f(t)|^{1/2}{\rm sgn}(\partial_t f(t))f(t),
        \end{equation}
        and 
        \begin{equation}\label{eq:nu}
            \nu(t):=3\int_{\R^2}|\partial_t f(t)|^{1/2}{\rm sgn}(\partial_t f(t))Af(t)\partial_tAf(t)-\frac{3}{2}\lambda(t)\int_{\R^2} |\partial_t f(t)|^{3/2}.
        \end{equation}
        To show that $\xi\in L^1(\R_+)$, we simply observe that
        \[\int_{\R^2}|\partial_t f(t)|^{1/2}{\rm sgn}(\partial_t f(t))f(t)\leq\|\partial_t f(t)\|_{3/2}^{1/2}\|f(t)\|_{3/2}, \]
        and the RHS is bounded. Moreover, it is evident that $\lambda'(t)$ is integrable.\\\\
        To show that $\nu\in L^\infty(\R_+)$, we need only show that the first term in \eqref{eq:nu} is bounded, as we have already shown that $\lambda(t)$ and $\|\partial_t f(t)\|_{3/2}$ are bounded. By H\"older's inequality, we have
        \begin{gather*}
            \int_{\R^2}|\partial_t f(t)|^{1/2}{\rm sgn}(\partial_t f(t))Af(t)\partial_tAf(t)
            \leq\|\partial_t f(t)\|_{3/2}^{1/2}\|Af(t)\|_3\|\partial_t Af(t)\|_3.
        \end{gather*}
        Of the three factors, only the last needs to be shown to be bounded. We have
        \begin{align*}
            \|\partial_t Af(t)\|_3&=2\|R^*[\cR f\cdot R\partial_t f]\|_3\lesssim\|\cR f\cdot R\partial_t f\|_{3/2}\\
            &\leq\|\cR f\|_3\|R\partial_t f\|_3\lesssim\|f(t)\|_{3/2}\|\partial_t f(t)\|_{3/2}.
        \end{align*}
\end{proof}

\medskip

\subsection{Sobolev Estimates, Fractional Integral Formulas and Precompactness of the Euler--Lagrange Flow}\label{subsec:sobolev}
Below in this subsection are several fractional-integral identities/estimates for $\cR$ and $\cR^*$. For a systematic treatment of these and more properties of the Radon transform, we refer the reader to Helgason's book \cite{Helgason}. We begin by proving two useful formulas. In this subsection, we only concern ourselves with non-negative radial functions. The first useful observation is that if $f$ non-negative and radial, then $Af$ is also radial and $\cR f(\alpha,s)$ is constant in the spherical variable $\alpha$. Moreover, we have
\begin{equation}
    \cR f(s)=2\int_s^\infty \frac{rf(r)}{\sqrt{r^2-s^2}}dr,\qquad Af(r)=4 \int_0^r  \frac{\cR f(s)^2}{\sqrt{r^2-s^2}}ds.
\end{equation}
The proof consists of two simple changes of variables. Regarding the first one, we observe that for any $\alpha\in \bbS^1,\, s\in\R$, each $x\in \R^2$ satisfying $x\cdot\alpha=s$ may be written as $x=s\alpha+t \alpha^\perp$ for a unique $t\in\R$, where $\alpha^\perp$ is the counter-clockwise rotation of $\alpha$ by $\pi/2$ radians. For each $t\in \R$, let $r=\sqrt{s^2+t^2}.$ Then, we have
\[\cR f(\alpha,s)=\int_{x\cdot\alpha=s}f(| x|)d\sigma(x)=2\int_{0}^\infty f\left(\sqrt{s^2+t^2}\right)dt=2\int_s^\infty\frac{rf(r)}{\sqrt{r^2-s^2}}dr.\]
For the second one, let $s=r\cos\theta$ and $g:=(\cR f)^2$. We have
\begin{align*}
    Af(x)=\cR^* g(x)=\int_{0}^{2\pi}g(r\cos \theta)\,d\theta=4 \int_0^{\pi/2}g(r\cos \theta)\,d\theta=4 \int_0^r \frac{g(s)}{\sqrt{r^2-s^2}}ds.
\end{align*}
This formula is extremely useful, as it will allow us to show that $A$ is smoothing of order $1/2$, which is essentially Theorem \ref{thm:A smoothing}.
\begin{proposition}\label{prop:A smoothing}
    Let $f$ be non-negative and radial, let $0<a<b$ and $\phi\in C_c^\infty(\R)$ be smooth cut-off function such that $\phi\equiv 1$ in $[a,b]$ and $\mathrm{supp}(\phi)\subseteq[a/2,2b]$. Then, $\phi\, Af\in H^{1/2,3/2}(\R)$ as a function of the radial variable $r$, and 
    \[\|\phi \,Af\|_{H^{1/2,3/2}(\R)}\lesssim_{a,b}\|f\|_{3/2}^2. \]
    Consequently, $\|Af\|_{H^{1/2,3/2}([a,b])}\lesssim_{a,b}\|f\|_{3/2}^2.$
\end{proposition}
\begin{proof}
    Set $g:= (\cR f)^2.$ We decompose $g$ as follows:
    \[g(s)=g(s)\chi_{[0,a/4]}(s)+g(s)\chi_{[a/4,2b]}(s)=:g_1(s)+g_2(s), \]
    and for $r\in[a/2,2b]$,
    \begin{equation}
        Af(r)=\int_0^{a/4} \frac{g_1(s)}{\sqrt{r^2-s^2}}ds+\int_{a/4}^r \frac{g_2(s)}{\sqrt{r^2-s^2}}ds.
    \end{equation}
    Moreover, for $r\in[a,2b]$ and $s\in(a/4,r)$ we have
    \[\frac{1}{\sqrt{r+s}}=\frac{1}{\sqrt{2s}}+(r-s)\psi(r,s),\]
    where
    \[\psi(r,s):=\left(\frac{1}{\sqrt{r+s}}-\frac{1}{\sqrt{2s}}\right)\frac{1}{r-s}=\frac{-1}{\sqrt{r+s}\sqrt{2s}\left(\sqrt{r+s}+\sqrt{2s}\right)}.\]
    Thus, we have
    \begin{align}
        \int_{a/4}^r \frac{g_2(s)}{\sqrt{r^2-s^2}}ds=\int_{a/4}^r\frac{(2s)^{-1/2}g_2(s)}{\sqrt{r-s}}ds+\int_{a/4}^r g_2(s)\sqrt{r-s}\,\psi(r,s)\,ds.
    \end{align}
    We write the full decomposition
    \begin{align}\label{eq:A decomp}
        Af(r)&=\int_0^{a/4} \frac{g_1(s)}{\sqrt{r^2-s^2}}ds+\int_{a/4}^r\frac{(2s)^{-1/2}g_2(s)}{\sqrt{r-s}}ds+\int_{a/4}^r g_2(s)\sqrt{r-s}\,\psi(r,s)\,ds\\
        &=:A_1(r)+A_2(r)+A_3(r). \nonumber
    \end{align}
    We will show that each piece $\phi(r)A_j(r)\in H^{1/2,3/2}(\R)$ and
    \[\|\phi(r)A_j(r)\|_{H^{1/2,3/2}(\R)}\lesssim_{a,b}\|f\|_{3/2}^2. \]
    Let $r\in\mathrm{supp}(\phi)\subset[a/2,2b]$. For $s\in[0,a/4]$, we have $\sqrt{r^2-s^2}\geq\frac{\sqrt{3}a}{4}$, and thus by H\"older's inequality we have
    \begin{align*}
        A_1(r)=\int_0^{a/4} \frac{g_1(s)}{\sqrt{r^2-s^2}}ds&\leq\frac{4}{\sqrt{3}a}\int_0^{a/4}g_1(s)ds\\
        &\leq\frac{4}{\sqrt{3}a}\|g_1\|_{3/2}(a/4)^{1/3}\lesssim a^{-2/3}\|f\|_{3/2}^2.
    \end{align*}
    Hence $\phi(r)A_1(r)\in L^\infty(\R).$ Since $\mathrm{supp}(\phi(r)A_1(r))\subseteq[a/2,2b]$, we also have $\phi(r)A_1(r)\in L^{3/2}(\R)$ and 
    \[\|\phi(r)A_1(r)\|_{L^{3/2}(\R)}\lesssim_{a,b}\|f\|_{3/2}^2. \]
    Moreover, for $s\in[0,a/4]$, the derivative $\partial_r\left[(r^2-s^2)^{-1/2}\right]$ is bounded as a function of $r\in\mathrm{supp}(\phi)\subset[a/2,2b]$. Thus,
    \[A_1'(r)=\partial_r\left[\int_0^{a/4} \frac{g_1(s)}{\sqrt{r^2-s^2}}ds\right]\lesssim_a \int_0^{a/4} g_1(s)ds\lesssim_a \|f\|_{3/2}^2.\]
    By the product rule, we have
    \begin{align*}
        \partial_r\left(\phi(r)A_1(r)\right)=\phi'(r)A_1(r)+\phi(r)A_1'(r).
    \end{align*}
    Since $\phi',A_1\in L^\infty(\R)$ and $\mathrm{supp}(\phi')\subseteq[a/2,2b]$, we have $\phi'(r)A_1(r)\in L^{3/2}(\R)$. Similarly, we have $\phi(r)A_1'(r)\in L^{3/2}(\R)$.
    Therefore, $$\phi(r)A_1(r)\in H^{1,3/2}(\R),$$ with norm $\lesssim_{a,b} \|f\|_{3/2}^2$. By the continuous embedding $$H^{1,3/2}([a,b])\hookrightarrow H^{1/2,3/2}([a,b]),$$ we have $$\phi(r)A_1(r)\in H^{1/2,3/2}([a,b]),$$ with norm $\lesssim_{a,b} \|f\|_{3/2}^2 $ .\\\\
    Next, we deal with the third term of \eqref{eq:A decomp}. 
    Since $\psi$ is bounded in its domain, we have
    \begin{equation*}
        A_3(r)=\int_{a/4}^r g_2(s)\sqrt{r-s}\,\psi(r,s)\,ds\lesssim_{a,b}\|g_2\|_{L^1([a/4,b])} \lesssim_{a,b}\|g_2\|_{3/2}\lesssim\|f\|_{3/2}^2,
    \end{equation*}
    and thus $\phi(r) A_3(r)\in L^\infty(\R)$ and $\|\phi(r)A_3(r)\|_\infty\lesssim\|f\|_{3/2}^2$. Since $\mathrm{supp}(\phi(r)A_3(r))\subseteq[a/2,2b]$, as before we conclude that $\phi(r) A_3(r)\in L^{3/2}(\R)$ and
    \[\|\phi(r)A_3(r)\|_{L^{3/2}(\R)}\lesssim_{a,b}\|f\|_{3/2}^2. \]
    To compute the derivative of $\phi(r) A_3(r)$, we first write
    \[A_3(r)=\int_{a/4}^b g_2(s)\sqrt{\max{\{r-s,0}\}}\psi(r,s)\,ds, \]
    and then differentiate (in the weak sense) to get
    \begin{align*}
        A_3'(r)&=\partial_r \int_{a/4}^b g_2(s)\sqrt{\max{\{r-s,0\}}}\psi(r,s)\,ds\\
        &=\int_{a/4}^r \frac{g_2(s)}{2\sqrt{r-s}}\psi(r,s)\, ds+\int_{a/4}^r g_2(s)\sqrt{r-s}\,\partial_r \psi(r,s)\,ds\\
        &\lesssim_{a,b}\int_{a/4}^b \frac{g(s)}{\sqrt{|r-s|}}ds+\|g_2\|_{L^1([a,b])}\\
        &\lesssim_{a,b}\int_{a/4}^b \frac{g_2(s)}{\sqrt{|r-s|}}ds+\|g_2\|_{L^{3/2}(\R)}\\
        &\lesssim_{a,b}\int_{\R} \frac{g_2(s)}{\sqrt{|r-s|}}ds+\|f\|_{3/2}^2.
    \end{align*}
    Thus, by H\"older's inequality we have
    \begin{align*}
        \|\phi(r)A_3'(r)\|_{L^{3/2}(\R)}&\lesssim_{a,b}\left\|\phi(r)\int_{\R} \frac{g_2(s)}{\sqrt{|r-s|}}ds\right\|_{L^{3/2}(\R)}+\|\phi(r)\|_{L^{3/2}(\R)}\|f\|_{3/2}^2 \\
        &\lesssim\|\phi\|_{L^2(\R)} \left\|\int_{\R} \frac{g_2(s)}{\sqrt{|r-s|}}ds\right\|_{L^{6}(\R)}+\|f\|_{3/2}^2 \\
        &\lesssim\|g_2\|_{L^{3/2}(\R)}+\|f\|_{3/2}^2\lesssim_{a,b}\|f\|_{3/2}^2,
    \end{align*}
    where in the last inequality we used the Hardy--Littlewood--Sobolev inequality. Furthermore,
    \[\|\phi'(r)A_3(r)\|_{L^{3/2}(\R)}\leq\|A_3(r)\|_{L^\infty(\R)}\|\phi'\|_{L^{3/2}(\R)}. \]
    Hence
    \[\left(\phi(r)A_3(r)\right)'=\phi'(r)A_3(r)+\phi(r)A_3'(r)\in L^{3/2}([a,b]),\]
    and so $\phi(r)A_3(r)\in H^{1,3/2}(\R)$ with norm $\lesssim_{a,b}\|f\|_{3/2}^2$. As before, we get
    $$ \phi(r)A_3(r)\in H^{1/2,3/2}(\R),$$ with norm $\lesssim_{a,b} \|f\|_{3/2}^2$.\\\\
    Lastly, we bound the second term of \eqref{eq:A decomp}. By H\"older's inequality we have
    \begin{align*}
        \|\phi(r)A_2(r)\|_{L^{3/2}(\R)}&\leq\|\phi\|_{2}\left\|\int_{a/4}^r\frac{s^{-1/2}g_2(s)}{\sqrt{r-s}}ds\right\|_{L^{6}(\R)}
        \lesssim\left\|\int_{a/4}^b\frac{s^{-1/2}g_2(s)}{\sqrt{|r-s|}}ds\right\|_{L^{6}(\R)}  ,
    \end{align*}
    which by the Hardy-Littlewood-Sobolev inequality is
    \begin{align*}
        \lesssim\|s^{-1/2}g_2(s)\|_{L^{3/2}(\R)}=\|s^{-1/2}g_2(s)\|_{L^{3/2}([a/4,b])} \lesssim_{a,b}\|g_2\|_{3/2}\lesssim\|f\|_{3/2}^2,
    \end{align*}
    since $g_2$ is supported in $[a/4,b]$. Now let 
    \[k(s):=|s|^{-1/2}\chi_{(0,\infty)}(s), \qquad \tilde{g}(s):= s^{-1/2}g_2(s). \]
    By definition, $\tilde{g}$ is supported in $[a/4,b]$. Hence, we have
    \begin{align*}
       A_2(r)= \int_{a/4}^r\frac{s^{-1/2}g_2(s)}{\sqrt{r-s}}ds&=\int_{a/4}^b\tilde{g}(s)|r-s|^{-1/2}\chi_{(0,\infty)}(r-s)ds\\
        &=\int_{\R}\tilde{g}(s)k(r-s)ds=k*\tilde{g}\,(r).
    \end{align*}
    We also have for $\sigma\neq0$
    \begin{align*}
        \widehat{k}(\sigma)=\int_0^\infty s^{-1/2}e^{-2\pi i\sigma s}ds&=|\sigma|^{-1/2}\int_0^\infty (|\sigma| s)^{-1/2}e^{-2\pi i\sigma s}|\sigma|ds\\
        &=|\sigma|^{-1/2}\int_0^\infty u^{-1/2}e^{-\mathrm{sgn}(\sigma)2\pi i u}du,
    \end{align*}
    where the last integral evaluates to $\frac{\sqrt{2}}{2}-\frac{\sqrt{2}}{2}i\,\mathrm{sgn}(\sigma)$ . Hence,
    \begin{align*}
        \widehat{\left((\partial_r)^{1/2}\left[k*\tilde{g}\right]\right)}(\sigma)=(2\pi|\sigma|)^{1/2}\widehat{k}(\sigma)\widehat{\tilde{g}}(\sigma)=c\left[1-i\,\mathrm{sgn}(\sigma)\right]\widehat{\tilde{g}}(\sigma).
    \end{align*}
    By taking the inverse Fourier transform, we obtain
    \begin{align*}
        (\partial_r)^{1/2}k*\tilde{g}(r)=c(I+H)\tilde{g}(r),
    \end{align*}
    where $I$ is the identity operator and $\widehat{Hf}(\sigma):=-i\mathrm{sgn}(\sigma)\widehat{f}(\sigma)$ is the Hilbert transform. Since both these operators are bounded on $L^{3/2}(\R)$ and $\tilde{g}$ is supported on $[a/4,b]$, we have
    \begin{align*}
        \|(\partial_r)^{1/2}A_2(r)\|_{L^{3/2}(\R)}&=\left\| (\partial_r)^{1/2}\int_{a/4}^r\frac{s^{-1/2}g_2(s)}{\sqrt{r-s}}ds\right\|_{L^{3/2}(\R)}\\
        &=\|(\partial_r)^{1/2}k*\tilde{g}\|_{L^{3/2}(\R)}\\
        &\leq c\|(I+H)\tilde{g}\|_{L^{3/2}(\R)}\lesssim\|\tilde{g}\|_{L^{3/2}(\R)}\\
        &=\|s^{-1/2}g_2(s)\|_{L^{3/2}([a/4,b])}\lesssim_{a,b}\|g_2\|_{L^{3/2}([a/4,b])}\lesssim\|f\|_{3/2}^2.
    \end{align*}
    By Theorem \ref{thm:Leibniz} (fractional Leibniz rule), we have
    \begin{align}
        &\|(\partial_r)^{1/2}\left[\phi(r)A_2(r)\right]\|_{L^{3/2}(\R)}\nonumber\\
        &\qquad\leq C\left[\|(\partial_r)^{1/2}\phi\|_{L^2(\R)}\|A_2\|_{L^6(\R)}+\|\phi\|_{\infty}\|(\partial_r)^{1/2}A_2\|_{L^{3/2}(\R)}.\right]
    \end{align}
    To see why the right-hand side of the above is finite, we note that $$A_2(r)=\int_{a/4}^b\frac{\tilde{g}(s)}{\sqrt{|r-s|}}ds=\int_{\R}\frac{\tilde{g}(s)}{\sqrt{|r-s|}}ds,$$
    and thus by the Hardy-Littlewood-Sobolev inequality, we get
    $$\|A_2\|_{L^6(\R)}\lesssim\|\tilde{g}\|_{L^{3/2}(\R)}=\|s^{-1/2}g_2(s)\|_{L^{3/2}([a/4,b])}\lesssim_{a,b}\|g_2\|_{L^{3/2}([a/4,b])}\lesssim\|f\|_{3/2}^2. $$
    Thus, $$\phi(r)A_2(r)\in H^{1/2,3/2}(\R),$$ with norm $\lesssim_{a,b}\|f\|_{3/2}^2$.\\\\
    By definition of the $H^{1/2,3/2}([a,b])$-norm and since $\phi\equiv1$ in $[a,b]$, we have
    \[\|Af\|_{H^{1/2,3/2}([a,b])}=\|\phi\, Af\|_{H^{1/2,3/2}([a,b])}\leq \|\phi\, Af\|_{H^{1/2,3/2}(\R)}\lesssim_{a,b}\|f\|_{3/2}^2. \]
\end{proof}
The following classical compactness result justifies the usefulness of the smoothing estimate in Proposition \ref{prop:A smoothing}.
\begin{theorem}[Triebel \cite{Triebel}, Theorem 3.2.3]\label{thm:Triebel}
    Let $0<a<b$. The space $H^{1/2,3/2}([a,b])$ embeds compactly into $L^r([a,b])$ for all $1\leq r<6$.
\end{theorem}
\begin{remark}
    Theorem $3.2.3$ in \cite{Triebel} is stated for general bounded Lipschitz domains $\Omega$ and the more general Triebel-Lizorkin spaces $F^s_{p,q}(\Omega)$, which we will not define here. Instead, we refer the interested reader to \cite{Triebel} for details. In the special case where $q=2$, the space $F^s_{p,q}(\Omega)$ reduces to $H^{s,p}(\Omega)$ (\cite{Triebel},\ Paragraph $3.2.2$). Hence Theorem \ref{thm:Triebel} is merely a special case of Theorem $3.2.3$ in \cite{Triebel} for $q=2$.
\end{remark}
Theorem \ref{thm:Triebel} has the following immediate corollary.
\begin{corollary}\label{cor:compactness}
 Let $0<a<b$. For every $r\in[1,6)$, any bounded sequence in $H^{1/2,3/2}([a,b])$ has a convergent subsequence (whose indices depend on $a,b,r$) in $L^r([a,b])$.    
\end{corollary}
\begin{remark}
    The general version of Theorem \ref{thm:Triebel} holds for both radial and non-radial functions. However, for our purposes, radiality is needed to prove Theorem \ref{thm:A smoothing}, without which we may not apply Theorem \ref{thm:Triebel}. We explain further in the following paragraph.
\end{remark}
Let $f(t)$ be the solution of \eqref{eq:EL} with non-negative initial data $f_0\in L^{3/2}(\R^2) $. If $f(t)$ is radial for all $t\geq0$, then Proposition \ref{prop:A smoothing} implies that $Af(t)$ is bounded in $H^{1/2,3/2}([a,b])$ for every $0<a<b$. In turn, Corollary \ref{cor:compactness} implies that $Af(t)$ has a convergent subsequence in $L^q([a,b])$ for every $q\in[1,6)$ and $0<a<b$, which allows us to show precompactness of $f(t)$ relatively easily. The following results show that if $f_0$ is radial, then so is $f(t)$ for all times $t>0$.

\begin{lemma}
    Let $\cG_\Phi(f):=A(f)^2-\lambda(f)f$ and let $O\in\R^{2\times 2} $ be a rotation. Denote $h_O(x):=h(O\cdot x).$ We have 
    \[\cG_\Phi(f_O)(x)=\cG_\Phi(f)_O(x).\]
\end{lemma}
\begin{proof}
    We first show $A(f_O)(x)=(Af)_O(x)$. We have
    \begin{align*}
        \cR (f_O)(\alpha,s)&=\int_{x\cdot \alpha=s}f_O(x)d\sigma(x)=\int_{(O\cdot x)\cdot(O\cdot\alpha)=s}f(O\cdot x)d\sigma(x)\\
        &=\int_{y\cdot(O\cdot \alpha)=s}f(y)d\sigma(y)=\cR f(O\cdot \alpha,s),
    \end{align*}
    and thus
    \begin{align*}
        A(f_O)(x)=\cR^*[(\cR f_O)^2](x)&=\int_{\bbS^1}\cR f_O(\alpha,x\cdot\alpha)^2d\alpha=\int_{\bbS^1}\cR f(O\cdot \alpha, x\cdot \alpha)^2d \alpha\\
        &=\int_{\bbS^1}\cR f(O\cdot\alpha,(O\cdot x)\cdot(O\cdot \alpha))^2 d\alpha\\
        &=\int_{\bbS^1} \cR f(\alpha,(O\cdot x)\cdot \alpha)^2d\alpha=Af(O\cdot x)=(Af)_O(x).
    \end{align*}
    Moreover, we have
    \[
    \lambda(f_O)=\frac{\int |f_O|^{1/2}(Af_O)^2}{\int |f_O|^{3/2}}=\frac{\int |f_O|^{1/2}(Af)_O^2}{\int |f_O|^{3/2}}=\frac{\int |f|^{1/2}(Af)^2}{\int |f|^{3/2}}=\lambda(f).
    \]
\end{proof}
\begin{proposition}\label{prop:radial}
    Let $f_0\in L^{3/2} $ be non-negative and radial. Then the solution $f(t)$ of \eqref{eq:EL} is (non-negative and) radial for all $t\geq 0$.
\end{proposition}
\begin{proof}
    Let $O\in\R^{2\times 2} $ be a rotation and let $f_O(x,t):=f(O\cdot x,t)$. We will show that $f_O$ also satisfies \eqref{eq:EL} with initial data $f_0$. Indeed, we have
    \[\partial_t f_O(x,t)=\partial_t f(O\cdot x,t)=\cG_\Phi(f(t))(O\cdot x)=\cG_\Phi(f(t))_O(x)=\cG_\Phi(f_O(t))(x),\]
    and since $f_0$ is radial,
    \[f_O(x,0)=f(O\cdot x,0)=f_0(O\cdot x)=f_0(x).\]
     By uniqueness of solutions we must have $f_O(t)=f(t)$ for all $t\geq0$. Since $O$ was an arbitrary rotation, we conclude that $f(t)$ is radial for all $t\geq0$.
\end{proof}
We now prove a preliminary convergence (precompactness) result for the flow \eqref{eq:EL} for non-negative radial initial data. For any $0<a<b$, define the annulus
$$S_{a,b}:=\{x\in\R^2:\ a< |x|< b\}.$$
\begin{theorem}\label{thm:raw convergence}
    Let $f(t)$ be the solution of \eqref{eq:EL} with non-negative radial initial data $f_0\in L^{3/2}(\R^2)$ of unit norm. Then there exists a strictly increasing sequence of times $t_n\to \infty$ and a non-negative radial function $f_\infty\in L^{3/2}(\R^2)$ such that $f_n:=f(t_n)\to f_\infty$ in $L^{3/2}(S_{a,b})$ for all $0<a<b$.
\end{theorem}
\begin{proof}
    Let $f(t)$ be the solution of \eqref{eq:EL} with non-negative radial initial data $f_0\in L^{3/2}(\R^2)$ and let $0<a<b$. By Proposition \ref{prop:A smoothing} and Lemma \ref{le:norm invariance} we have $$\|Af(t)\|_{H^{1/2,3/2}([a,b])}\lesssim_{a,b}\|f(t)\|_{3/2}^2=\|f_0\|_{3/2}^3.$$ Hence, by Corollary \ref{cor:compactness} for $r=3\in[1,6)$, there exists a sequence $\tau_n$ depending on $a,b$ such that $Af(\tau_n)$ converges to a non-negative function $G_{a,b}$ in $L^3([a,b])$. Since the domain $[a,b]$ is bounded, we may upgrade $1$-dimensional convergence of $Af(\tau_n)$ in $L^{3}([a,b],dr)$ to $2$-dimensional convergence in $L^{3}(S_{a,b}, dx)$. Indeed, we have
\begin{align*}
    \int_{S_{a,b}}\left|Af(\tau_n)-G_{a,b}\right|(x)^{3}dx&\sim\int_a^b\left|Af(\tau_n)-G_{a,b}\right|(r)^{3}r\,dr\\
    &\leq b\int_a^b\left|Af(\tau_n)-G_{a,b}\right|(r)^{3}dr\\
    &\sim b\|Af(\tau_n)-G_{a,b}\|_{L^3([a,b])}^3 \to 0,
\end{align*}
as $n\to \infty$. We claim that $(Af(\tau_n))^2$ converges to $G_{a,b}^2$ in $L^{3/2}(S_{a,b})$. Indeed, by Cauchy--Schwarz we have
\begin{align*}
    \|(Af(\tau_n))^2-G_{a,b}^2\|_{L^{3/2}(S_{a,b})}&=\|\left(Af(\tau_n)+G_{a,b}\right)(Af(\tau_n)-G_{a,b})\|_{L^{3/2}(S_{a,b})}\\
    &\leq\|Af(\tau_n)+G_{a,b}\|_{L^{3}(S_{a,b})}\|Af(\tau_n)-G_{a,b}\|_{L^{3}(S_{a,b})},
\end{align*}
and 
\begin{align*}
    \|Af(\tau_n)+G_{a,b}\|_{L^{3}(S_{a,b})}&\leq\|Af(\tau_n)\|_{L^3(S_{a,b})}+\|G_{a,b}\|_{L^3(S_{a,b})}\\
    &\lesssim\|f(\tau_n)\|_{L^{3/2}(\R^2)}^2 +\|G_{a,b}\|_{L^3(S_{a,b})}\\
    &=\|f_0\|_{L^{3/2}(\R^2)}^2 +\|G_{a,b}\|_{L^3(S_{a,b})}.
\end{align*}
Moreover, by Proposition \ref{prop:extr_conv}, $A(f(\tau_n))^2-\lambda(f(\tau_n))f(\tau_n)$ converges to $0$ in $L^{3/2}(\R^2)$. In particular, it converges to $0$ in $L^{3/2}(S_{a,b})$. Since $\lambda(f(t))$ is non-decreasing and bounded, it also converges to some $\lambda_\infty>0$ independently of $a,b$. Therefore, we have
\[ f(\tau_n)=\frac{A(f(\tau_n))^2+\left[\lambda(f(\tau_n))f(\tau_n)-A(f(\tau_n))^2\right]}{\lambda(f(\tau_n))}\to \frac{G_{a,b}^2}{\lambda_\infty},\]
in $L^{3/2}(S_{a,b})$.\\\\
We now extract a sequence that converges in $L^{3/2}(S_{a,b})$ for any $0<a<b$, by a diagonal Cantor argument. For each $k\in\mathbb{N}$, let $$S_k:=\{1/k\leq|x|\leq k\}.$$ By applying the above argument iteratively, we may find nested subsequences $$(t_n^{(1)})_{n\in\mathbb{N}}\supseteq\cdots\supseteq(t_n^{(k)})_{n\in\mathbb{N}}\supseteq(t_n^{(k+1)})_{n\in\mathbb{N}}\supseteq\cdots,$$ 
such that for each $k\in\mathbb{N}$, the sequence $f_n^{(k)}:=f(t_n^{(k)})$ converges (as $n\to \infty$) to a limit function $f^{(k)}$ in $L^{3/2}(S_k)$.
Since the subsequences are nested, for every $k\in\mathbb{N}$, uniqueness of the limit in $L^{3/2}(S_k)$ guarantees that $f^{(k)}\equiv f^{(k+1)}$ in $S_k$. Hence, the function $f_\infty:\R^2\setminus\{0\}\to \R$ such that $f_{\infty}\equiv f^{(k)}$ in $S_k$ for all $k\in\mathbb{N}$, is well-defined. We extend $f_\infty$ to $\R^2$ by setting $f_\infty(0)=1$. Moreover, define 
\[t_n:=t_n^{(n)},\qquad f_n:=f(t_n).\]
Then, for any $0<a<b$ we have
\[\lim_{n\to \infty}\|f_n-f_\infty\|_{L^{3/2}(S_{a,b})}=0. \]
Indeed, let $0<a<b$ and choose $k\in\mathbb{N}$ such that $[a,b]\subset[1/k,k]$. By the nesting property of the sequences $(t_n^{(k)})_n$, for every $k\in\mathbb{N},$ $(t_n)_{n\geq k}$ is a subsequence of $(t_m^{(k)})_m$. Hence $f_n\to f^{(k)}\equiv f_\infty$ as $n\to \infty$ in $L^{3/2}(S_k)$. Since $S_{a,b}\subset S_k$, the claim follows. Moreover, a limit of non-negative radial functions is itself non-negative and radial. This may be verified by extracting a pointwise-convergent subsequence in every annulus. \\\\
We show that $f_\infty\in L^{3/2}(\R^2)$. For any $0<a<b$ we have
\begin{align*}
    \int_{S_{a,b}}f_\infty(x)^{3/2}dx= \lim_{n\to \infty} \int_{S_{a,b}}f_n(x)^{3/2}dx\leq\limsup_{n}\int_{\R^2}f_n^{3/2}= 1.
\end{align*}
We take the limit as $a\to 0, b\to \infty$ to obtain
\[\int_{\R^2}f_\infty^{3/2}\leq 1. \]
\end{proof}
\begin{remark}\label{remark:possible vanishing}
    Notice that Theorem \ref{thm:raw convergence} does not assert that the limit $f_\infty$ is non-trivial. As we will see in Subsection \ref{subsec:convergence}, this will become the case after rescaling $f(t)$ appropriately.
\end{remark}

\subsection{Symmetry-Breaking and Non-Vanishing After Rescaling}\label{subsec:convergence}
By Theorem \ref{thm:raw convergence}, the solution $f(t)$ of \eqref{eq:EL} with non-negative, radial initial data has a time-subsequence $f(t_n)$ that converges locally in $L^{3/2}(a\leq|x|\leq b)$ for every $0<a<b$. To recover a non-trivial limit, we need an additional mechanism that prevents loss of precompactness. In PDE theory, this mechanism is often provided by tightness estimates of the following form. Let $u:\R^n\times \R_+\to \R$ be such that $\|u(t)\|_{L^p(\R^n)}=1$ for all $t\geq0$. We say $u$ is \emph{tight} iff for every $\varepsilon>0$ there exist $R>r>0$ such that
\begin{equation}\label{eq:tightness}
    \sup_t \int_{\{|x|\leq r\}\cup\{|x|\geq R\}} |u(x,t)|^{p}dx <\varepsilon,
\end{equation}
or equivalently
\begin{equation}\label{eq:tightness 2}
    \inf_t \int_{\{r\le |x|\le R\}} |u(x,t)|^{p}dx \geq1-\varepsilon.
\end{equation}
Tightness ensures that the $L^p$-mass of $u$ cannot spread out or transport to infinity, or concentrate near the origin too fast, such that $u(t)$ ends up converging weakly to $0$ or to a Dirac mass. \\\\
However, many PDEs that are invariant under non-compact symmetry groups (such as translations, dilations and shears) lack tightness. Even in the presence of a conservation law (such as mass conservation), convergence to a non-zero stationary state can fail in dynamical systems (e.g. heat/Schr\"odinger equation). A common way to deal with this issue is to break symmetry by fixing an affine frame (e.g. centre and rescale at every time $t$), which ensures that the affine-transformed trajectory is precompact in $L^p_{\mathrm{loc}}(\R^2)$. It is the precompactness property of the rescaled dynamics which yields non-trivial limits in $L^p_{\mathrm{loc}}(\R^2)$. In this subsection, we show that this is indeed the case for \eqref{eq:EL} --- at least for non-negative, radial initial data.\\\\
As we have seen in Subsection \ref{subsec:sobolev}, a convenient (but perhaps not necessary; see Section \ref{sec:non-radial}) assumption that rules out a large number of non-compact symmetries is \emph{radiality}. Indeed, if $f(t)$ is radial for all times $t\geq 0$, then three main symmetry groups are ruled out:
\begin{itemize}
    \item \textbf{Translations:} By fixing the centre of mass at the origin, we ensure that mass cannot be transported to infinity.
    \item \textbf{Anisotropic Dilations:} Radiality eliminates stretching unevenly in different directions.
    \item \textbf{Rotations:} These are automatically eliminated, since radial functions are invariant under rotations.
\end{itemize}
Thus, the only remaining symmetry is isotropic dilations, which we deal with by rescaling the solution appropriately at each time $t$. In what follows, $f(t)$ will denote the solution of \eqref{eq:EL} with non-negative, radial initial data $f_0\in L^{3/2}(\R^2)$ of unit norm.
\begin{theorem}\label{thm:rescaled convergence}
    Let $\varepsilon\in(0,10^{-2})$ and suppose that the initial data $f_0$ satisfies 
    \begin{equation}\label{eq:trapping}
        \Phi(f_0)^3\geq \|\cR\|^3(1-\varepsilon).
    \end{equation}
    Then there exist sequences of times and scales $t_n,\mu_n>0$, respectively, such that $F_n:= \mu_n^{4/3}f_{\mu_n}(t_n)$ converges in $L^{3/2}(S_{a,b})$ for all $0<a<b$ to a \emph{non-trivial} limit function $F_\infty\in L^{3/2}(\R^2)$ which is independent of  $a,b$. Furthermore, $(AF_n)^2\to(AF_\infty)^2$ in $L^{3/2}(S_{a,b})$, for all $0<a<b$.
\end{theorem}
\begin{remark}
    The condition \eqref{eq:trapping} is a statement about quantitative closeness of $f_0$ to the extremiser manifold $\cE$. In fact, in combination with Theorem \ref{thm:Drouot} it implies that  $d(f_0,\cE)<\varepsilon$. A further discussion on whether this condition is necessary for non-triviality of the limit may be found in Section \ref{sec:energy threshold}.
\end{remark}
\begin{remark}
    As the reader might expect, the difference between the conclusions of Theorems \ref{thm:raw convergence} and \ref{thm:rescaled convergence} is the non-triviality assertion in the latter.
\end{remark}
Before proving Theorem \ref{thm:rescaled convergence}, we prove several lemmas related to condition \eqref{eq:trapping} that will be used in the proof. The following two lemmas combined essentially assert that if a function is near the extremiser-manifold $\cE$, then its mass cannot simultaneously concentrate in two disjoint regions whose supports are well-separated. Precisely,
\begin{lemma}\label{lemma:two scale}
    Let $\varepsilon\in(0,10^{-2})$, let $c>0$ be an absolute constant (the value of which is evident from the proof) and let $c\sigma^{1/3}<\varepsilon$. Let $f\in L^{3/2}(\R^2)$ be non-negative, radial with $\|f\|_{3/2}=1$ and suppose 
    \[\Phi(f)^3\ge \|\cR\|^3(1-\varepsilon).\]
    For $r>0$ define $G_f(r):=\int_{|x|\leq r}f(x)^{3/2}dx.$ Then, we have
    \begin{equation}\label{eq:two scale}
        G_f(r)\left[1-G_f(r/\sigma)\right]\leq2\varepsilon+c\left[G_f(r/\sigma)-G_f(r)\right]^{2/3},
    \end{equation}
    for all $r>0$.
\end{lemma}
\begin{proof}
    Fix $r>0$ and let $$f_1:=f\, \chi_{|x|\geq \frac{r}{\sigma}},\qquad f_2:=f\,\chi_{|x|\leq r},\qquad f_3:=f-f_1-f_2.$$
    Let also
    \begin{gather*}
        \alpha:=\int_{\R^2} f_2^{3/2}=G_f(r),\quad \beta:=\int_{\R^2} f_1^{3/2}=1-G_f(r/\sigma),\\
        m:=\int_{\R^2}f_3^{3/2}=G_f(r/\sigma)-G_f(r) . 
    \end{gather*} 
    Then
    \[\alpha+\beta+m=1,\]
    and \eqref{eq:two scale} becomes
    \begin{equation}\label{eq:two scale v2}
        \alpha\beta\leq 2\varepsilon+c m^{2/3}.
    \end{equation}
    We have
    \begin{align}\label{eq:decomp}
        \|\cR\|^3(1-\varepsilon)\leq\Phi(f)^3=\int (\cR f)^3&=\int (\cR f_1)^3+\int (\cR f_2)^3+\int (\cR f_3)^3\\
        &+3\int (\cR f_1)^2 \cR f_2+3\int\cR f_1(\cR f_2)^2\nonumber\\
        &+3\int(\cR f_3)^2(\cR f_1+\cR f_2)\\
        &+3\int\cR f_3\left[(\cR f_1)^2+(\cR f_2)^2\right] \nonumber\\
        &+6\int \cR f_1\cR f_2 \cR f_3. \nonumber
    \end{align}
    We have
    \begin{align*}
        \int (\cR f_1)^3+\int (\cR f_2)^3+\int (\cR f_3)^3\leq\|\cR\|^3[\alpha^2+\beta^2+m^2],
    \end{align*}
    and
    \begin{align}\label{eq:m bound}
        &3\int(\cR f_3)^2(\cR f_1+\cR f_2)+3\int\cR f_3\left[(\cR f_1)^2+(\cR f_2)^2\right]
        +6\int \cR f_1\cR f_2 \cR f_3\\
        &\leq3(\alpha^{2/3}+\beta^{2/3})m^{4/3}+3m^{2/3}(\alpha^{4/3}+\beta^{4/3})+6\alpha^{2/3}\beta^{2/3}m^{2/3}\nonumber\\
        &\leq 100 m^{2/3}< cm^{2/3},
    \end{align}
    since $\alpha,\beta,m\leq 1$. Moreover, since $ \mathrm{supp}(f_2)\subseteq B(0,r)$, we have $\mathrm{supp}(\cR f_2)\subseteq[-r,r]$. Hence, the products $(\cR f_1)^2\cR f_2$ and $(\cR f_2)^2\cR f_1$ are also supported in $[-r,r]$. Thus, we have
    \begin{align*}
        &3\int_{|s|\leq r} (\cR f_1)^2 \cR f_2+3\int_{|s|\leq r} \cR f_1(\cR f_2)^2\\
        \leq &3\|\cR f_1\|_{L^\infty([0,r])}^2\|\cR f_2\|_{1}+3\|\cR f_1\|_{L^\infty([0,r])}\|\cR f_2\|_2^2\\
        \leq &3\|\cR f_1\|_{L^\infty([0,r])}^2\|\cR f_2\|_{2}r^{1/2}+3\|\cR f_1\|_{L^\infty([0,r])}\|\cR f_2\|_2^2
    \end{align*}
    We have $\mathrm{supp}(f_2)\subseteq B(0,r)$ and $\cR$ is bounded from $L^{4/3}(\R^2)$ to $L^2(\R^2)$ with operator norm $C:=C_{L^{4/3}\to L^2}>0$ with a constant related to the optimal constant in the Hardy--Littlewood--Sobolev inequality. A proof of this fact may be found in Appendix \ref{ap:Radon bound}. Thus, we have 
    \begin{align*}
        \|\cR f_2\|_2\leq C\|f_2\|_{L^{4/3}(|x|\leq r)}\leq C \|f_2\|_{3/2}r^{1/6}\leq C r^{1/6},
    \end{align*}
    and since $\mathrm{supp}(f_1)\subseteq\{|x|\geq r/\sigma\}$, for $|s|\leq r$ we have
    \begin{align*}
        \cR f_1(s)=\int_{\sqrt{s^2+t^2}=r/\sigma}^{t=\infty} f(\sqrt{s^2+t^2})dt=\int_{r/\sigma}^\infty \frac{uf(u)}{\sqrt{u^2-s^2}}du.
    \end{align*}
    We let $d\mu(u):=u\,du$ and apply H\"older's inequality to the last integral with respect to the measure $d\mu(u)$ to get
    \begin{align*}
        \cR f_1(s)&\leq\left(\int_{r/\sigma}^\infty f_1(u)^{3/2}d\mu(u)\right)^{2/3}\left(\int_{r/\sigma}^\infty \frac{d\mu(u)}{\sqrt{u^2-s^2}^3}\right)^{1/3}\\
        &\leq \|f_1\|_{L^{3/2}(\R^2)}\left(\frac{1}{2}\int_{r/\sigma}^\infty \frac{2u\,du}{(u^2-s^2)^{3/2}}\right)^{1/3}\\
        &\leq\left(\left[-(u^2-s^2)^{-1/2}\right]_{r/\sigma}^\infty\right)^{1/3}=\left(\frac{1}{\sqrt{(r/\sigma)^2-s^2}}\right)^{1/3}\\
        &\leq\sigma^{1/3}\left(\frac{1}{\sqrt{r^2-(\sigma r)^2}}\right)^{1/3}\leq \frac{\sigma^{1/3}}{r^{1/3}}\left(\frac{1}{\sqrt{1-\sigma^2}}\right)^{1/3}\leq 2\frac{\sigma^{1/3}}{r^{1/3}}.
    \end{align*}
    Therefore,
    \begin{align}\label{eq:sigma bound}
        3\int (\cR f_1)^2 \cR f_2+3\int\cR f_1(\cR f_2)^2&\leq 12C \frac{\sigma^{2/3}}{r^{2/3}}r^{1/6}r^{1/2}+6C^2\frac{\sigma^{1/3}}{r^{1/3}}r^{1/3}\\
        &=(12C+6C^2)\sigma^{1/3}. \nonumber
    \end{align}
    Inequalities, \eqref{eq:m bound}, \eqref{eq:sigma bound} and \eqref{eq:decomp} yield
    \begin{align*}
        \|\cR\|^3(1-\varepsilon)&\leq\|\cR\|^3[\alpha^2+\beta^2+m^2]+c\sigma^{1/3}+cm^{2/3}\\
        &\leq\|\cR\|^3[(1-m)^2-2\alpha\beta+m^2]+(12C+6C^2)\sigma^{1/3}+100 m^{2/3}.
    \end{align*}
    Dividing by $\|\cR\|^3$ and rearranging yields
    \begin{align*}
        2\alpha\beta\leq\varepsilon +2(m^2-m)+\frac{12C+6C^2}{\|\cR\|^3}\sigma^{1/3}+\frac{100}{\|\cR\|^3}m^{2/3},
    \end{align*}
    and since $m\leq 1$, we have $m^2\leq m$. Hence, by choosing $c>\max\left\{\frac{12C+6C^2}{\|\cR\|^3},\frac{100}{\|\cR\|^3}\right\}$, we obtain
    \[\alpha\beta\leq 2\varepsilon+cm^{2/3},\]
    which is \eqref{eq:two scale v2}.
\end{proof}
\begin{lemma}\label{lemma:mass trapping}
    Let $\varepsilon\in(0,10^{-2})$, $f\in L^{3/2}(\R^2)$ be non-negative, radial with $\|f\|_{3/2}=1$. Let $c>0$ be the absolute constant in the statement of Lemma \ref{lemma:two scale} and let $c\sigma^{1/3}\leq \varepsilon$. For $r>0$ define
    \[m(r):=\int_{r\leq|x|\leq r/\sigma}f(x)^{3/2}dx=G_f(r/\sigma)-G_f(r).\]
    Suppose
    \begin{equation}
        \Phi(f)^3\geq \|\cR\|^3(1-\varepsilon).
    \end{equation}
    Then there exists a constant $m_*=m_*(\varepsilon)>0$ independent of $f$ and a scale $\mu>0$ depending on $f$ such that $m(\mu)\geq m_*.$
\end{lemma}
\begin{proof}
    Suppose for a contradiction that $m(r)\leq \frac{\varepsilon^{3/2}}{c^{3/2}}$ for all $r>0$, where $c$ is the constant from Lemma \ref{lemma:two scale}. Then, Lemma \ref{lemma:two scale} yields
    \begin{equation}\label{eq:dichotomy}
        G_f(r)(1-G_f(r/\sigma))\leq 2\varepsilon+m(r)^{2/3}\leq 3\varepsilon,
    \end{equation}
    for all $r>0$. Let 
    \[\mu:=\sup\left\{r>0: G_f(r)<\sqrt{3\varepsilon}\right\}.\]
    Since $\lim_{r\to \infty}G_f(r)=\int_{\R^2}f^{3/2}=1,$ we must have $\mu<\infty$. Thus, for all $r>\mu$ we have $G_f(r)\geq\sqrt{3\varepsilon}$ and \eqref{eq:dichotomy} implies
    \[G_f(r/\sigma)>1-\sqrt{3\varepsilon},\]
    for all $r>\mu$. Since $G_f(r)$ is continuous in $r$, taking the limits as $r\to\mu^-$ and $r\to \mu^+$ yields
    \[G_f(\mu)\leq \sqrt{3\varepsilon},\qquad G_f(\mu/\sigma)\geq1-\sqrt{3\varepsilon},\]
    respectively and thus
    \[m(\mu)=G_f(\mu/\sigma)-G_f(\mu)\geq 1-2\sqrt{3\varepsilon},\]
    which contradicts $m(\mu)\leq \varepsilon^{3/2}$, since $\varepsilon<1/100$.
\end{proof}

We now prove Theorem \ref{thm:rescaled convergence}. Since $\Phi(f(t))$ is non-decreasing, we have 
\[\Phi(f(t))^3\geq\|\cR\|^3(1-\varepsilon),\]
for all $t\geq 0$. Thus, by Lemma \ref{lemma:mass trapping}, there exists a constant $M_*>0$ with the following property: for each $t>0$ there exists a scale $\mu_t>0$ such that 
\[m(\mu_t)=\int_{\mu_t\leq |x|\leq \mu_t/\sigma} f(x,t)^{3/2}dx\geq M_*. \]
Let $F(x,t):=\mu_t^{4/3}f_{\mu_t}(x,t)$. By Lemma \ref{le:norm invariance} we have
\[\|F(t)\|_{L^{3/2}(\R^2)}=\|\mu_t^{4/3}f_{\mu_t}(t)\|_{L^{3/2}(\R^2)}=\|f(t)\|_{L^{3/2}(\R^2)}=\|f_0\|_{L^{3/2}(\R^2)}, \]
and
\begin{align}\label{eq:mass trapping}
   \nonumber\int_{1\leq|x|\leq\sigma^{-1}}F(x,t)^{3/2}dx&=\int_{1\leq|x|\leq\sigma^{-1}}\mu_t^{2}f(\mu_tx,t)^{3/2}dx\\
   &=\int_{\mu_t\leq|x|\leq\mu_t/\sigma}f(x,t)^{3/2}dx>M_*.
\end{align}
Inequality \eqref{eq:mass trapping} states a key property of the rescaled sequence $F(t)$: a \emph{non-zero constant proportion} of the $L^{3/2}$-mass is trapped inside the fixed annulus $1\leq|x|\leq\sigma^{-1}$ for all times $t\geq 0$. As we will show, this property ensures that  $F(t)$ cannot concentrate or spread to $0$. We have
\[\|AF(t)\|_{H^{1/2,3/2}([a,b])}\lesssim_{a,b}\|F(t)\|_{L^{3/2}(\R^2)}^2, \]
and
\[\|AF(t)^2-\lambda(F(t))F(t)\|_{L^{3/2}(\R^2)}= \|Af(t)^2-\lambda(f(t))f(t)\|_{L^{3/2}(\R^2)},\]
for all $t\geq0,\ b>a>0$. To justify the last identity, we simply remark that for any $\mu>0$ and $g\geq0$ we have
\[A(\mu^{4/3}g_\mu)=\mu^{2/3}(Ag)_\mu,\qquad \lambda(\mu^{4/3}g_\mu)=\lambda(g).\]
The reader may verify that the proof of Theorem \ref{thm:raw convergence} applies verbatim to $F(t)$ and thus yields a sequence of times $t_n$ and a non-negative radial function $F_\infty\in L^{3/2}(\R^2)$ with $\|F_\infty\|_{L^{3/2}(\R^2)}\leq 1$ such that $$F_n:=F(t_n)\to F_\infty$$ in $L^{3/2}(S_{a,b})$ for all $0<a<b$. Furthermore, property \eqref{eq:mass trapping} yields
\[\int_{1\leq|x|\leq \sigma^{-1}}F_\infty(x)^{3/2}dx=\lim_{n\to \infty}\int_{1\leq|x|\leq\sigma^{-1}} F_n(x)^{3/2}dx\geq M_*>0, \]
and thus $F_\infty\neq 0$.\\\\
It remains to show that $(AF_n)^2\to (AF_\infty)^2$ in $L^{3/2}(a\leq|x|\leq b)$. By Cauchy--Schwarz,
\[\|(AF_n)^2-(AF_\infty)^2\|_{L^{3/2}(a\leq|x|\leq b)}\lesssim_{a,b}\|AF_n+AF_\infty\|_{L^{3}(a\leq|x|\leq b)} \|AF_n-AF_\infty\|_{L^{3}(a\leq|x|\leq b)}.\]
Moreover,
\begin{align*}
    \|AF_n+AF_\infty\|_{L^{3}(a\leq|x|\leq b)}^3\lesssim\|F_n\|_{L^{3/2}(\R^2)}+\|F_\infty\|_{L^{3/2}(\R^2)}\leq 2.
\end{align*}
Hence, it will suffice to show $AF_n\to AF_\infty$ in $L^{3}(a\leq|x|\leq b)$ for all $0<a<b$.\\\\
Fix $0<a<b$. We first prove a useful identity. For $f,g\geq0$ we have
\[Af-Ag=\cR^*\left[(\cR f)^2-(\cR g)^2\right]=\cR^*\left[ \cR(f+g)\, \cR(f-g)\right]=B(f+g,f-g),\]
where $B(f,g)=\cR^*(\cR f\, \cR g)$ is defined in Lemma \ref{lemma:Cauchy-Schwarz}. Furthermore, by Lemma \ref{lemma:Cauchy-Schwarz} and standard Cauchy--Schwarz, we have
\begin{align*}
    \|AF_n-AF_\infty\|_{L^{3}(a\leq|x|\leq b)}&=\|B(F_n+F_\infty,F_n-F_\infty)\|_{L^{3}(a\leq|x|\leq b)} \\
    &\leq \|A(F_n+F_\infty)^{1/2}A(F_n-F_\infty)^{1/2}\|_{L^{3}(a\leq|x|\leq b)}\\
    &\leq\|A(F_n+F_\infty)\|_{L^{3}(a\leq|x|\leq b)}^{1/2}\|A(F_n-F_\infty)\|_{L^{3}(a\leq|x|\leq b)}^{1/2}.
\end{align*}
The first factor on the right-hand side of the last inequality is bounded. Indeed, we have
\begin{align*}
    \int_{a\leq |x|\leq b}A(F_n+F_\infty)(x)^{3}dx
    \lesssim_a\|A(F_n+F_\infty)\|_{L^{3}(\R^2)}^{3}\lesssim\|F_n+F_\infty\|_{L^{3/2}(\R^2)}^{3}.
\end{align*}
Hence, it will suffice to show that
\[\|A(F_n-F_\infty)\|_{L^{3}(a\leq|x|\leq b)}\to 0.\]
To that end, let $R>0$ be a large radius satisfying $R>2b$ and $R^{-1}<a/2$ (but otherwise free to choose later) and let $h_n:=F_n-F_\infty$. We write
\begin{align*}
    h_n&= h_n\chi_{|x|< R^{-1}}+ h_n\chi_{R^{-1}\leq |x|\leq R}+ h_n\chi_{|x|>R}\\
    &=:h_n^{(1)}+h_n^{(2)}+h_n^{(3)}.
\end{align*}
Again by Lemma \ref{lemma:Cauchy-Schwarz}, we have
\begin{align*}
    |Ah_n|&=\left|\cR^*\left(\left[\cR h_n^{(1)}+\cR h_n^{(2)} +\cR h_n^{(3)}\right]^2\right)\right|\\
    &\leq\sum_{i=1}^3A(h_n^{(i)})+2\sum_{i< j}\left|\cR^*\left(\cR h_n^{(i)}\cR h_n^{(j)}\right)\right|\\
    &\leq\sum_{i=1}^3A(h_n^{(i)})+2\sum_{i<j}A(h_n^{(i)})^{1/2}A(h_n^{(j)})^{1/2},
\end{align*}
pointwise a.e. Therefore,
\begin{align*}
    \|A|h_n|\|_{L^{3}(a\leq|x|\leq b)}&\leq \sum_{i=1}^3 \|A|h_n^{(i)}|\,\|_{L^{3}(a\leq|x|\leq b)}\\
    &\qquad+2\sum_{i<j}\|A(|h_n^{(i)}|)^{1/2}A(|h_n^{(j)}|)^{1/2}\|_{L^{3}(a\leq|x|\leq b)}\\
    &\leq \sum_{i=1}^3 \|A|h_n^{(i)}|\,\|_{L^{3}(a\leq|x|\leq b)}\\
    &\qquad+2\sum_{i<j}\|A|h_n^{(i)}|\,\|_{L^{3}(a\leq|x|\leq b)}^{1/2}\|A|h_n^{(j)}|\,\|_{L^{3}(a\leq|x|\leq b)}^{1/2},
\end{align*}
and thus it will suffice to show
\begin{equation}\label{eq:A error vanish}
    \sum_{i=1}^3 \|A|h_n^{(i)}|\,\|_{L^{3}(a\leq|x|\leq b)}\to 0,\ \mbox{ as }\ n\to \infty.
\end{equation}
Since $  \supp(h_n^{(1)})\subseteq\{|x|<R^{-1}\}$, we have $\mathrm{supp}[\cR(h_n^{(1)})]\subseteq[-R^{-1},R^{-1}]$. Hence, for $r\in[a,b]$ we have
\begin{align*}
    A(|h_n^{(1)}|)(r)&=4\int_0^{r}\frac{\cR (|h_n^{(1)}|)(s)^2}{\sqrt{r^2-s^2}}ds=4\int_0^{R^{-1}}\frac{\cR (|h_n^{(1)}|)(s)^2}{\sqrt{r^2-s^2}}ds\\
    &\lesssim \frac{1}{a}\int_0^{R^{-1}}\cR (|h_n^{(1)}|)(s)^2ds
    \lesssim_a \left(\int_0^{R^{-1}} \cR(|h_n^{(1)}|)(s)^3 ds\right)^{2/3} R^{-1/3}\\
    &\lesssim\|h_n^{(1)}\|_{L^{3/2}(\R^2)}^2R^{-1/3}\leq\|F_n-F_\infty\|_{L^{3/2}(\R^2)}^2R^{-1/3}\lesssim R^{-1/3}.
\end{align*}
Hence, for the first term in \eqref{eq:A error vanish} we have
\begin{equation}\label{eq:h_1 small}
    \|A(h_n^{(1)})\|_{L^{3}(a\leq|x|\leq b)}\lesssim_{a,b}R^{-1/3}.
\end{equation}
Since $\supp(h_n^{(3)})\subseteq[R,\infty)$, for $s\in[a,b]\subseteq[0,R/2]$ we have
\begin{align*}
    \cR( h_n^{(3)})= 2\int_s^\infty\frac{r h_n^{(3)}(r)}{\sqrt{r^2-s^2}}dr=2\int_R^\infty\frac{r h_n^{(3)}(r)}{\sqrt{r^2-s^2}}dr,
\end{align*}
and by H\"older's inequality with respect to the measure $r\, dr$, we have
\begin{align*}
    \cR (|h_n^{(3)}|)(s)&\lesssim\left(\int_R^\infty |h_n^{(3)}(r)|^{3/2}r\,dr\right)^{2/3}\left(\int_R^\infty\frac{r\,dr}{(r^2-s^2)^{3/2}}\right)^{1/3}\\
    &=\left(\int_R^\infty |h_n^{(3)}(r)|^{3/2}r\,dr\right)^{2/3}\left(\frac{1}{2}\int_{R^2}^\infty\frac{du}{(u-s^2)^{3/2}}\right)^{1/3}\\
    &\leq \|F_n-F_\infty\|_{L^{3/2}(\R^2)}\left(\left[-(u-s^2)^{-1/2}\right]_{u=R^2}^{u=\infty}\right)^{1/3}\\
    &=\|F_n-F_\infty\|_{L^{3/2}(\R^2)}\frac{1}{(R^2-s^2)^{1/6}}\lesssim R^{-1/3}.
\end{align*}
Hence, for $r\in[a,b]$ we have
\begin{align*}
    A(|h_n^{(3)}|)(r)=4\int_0^{\pi/2}\cR (|h_n^{(3)}|)(r\cos\theta)^2\,d\theta\lesssim R^{-2/3},
\end{align*}
which yields
\begin{equation}\label{eq:h_3 small}
    \|A(|h_n^{(3)}|)\|_{L^{3}(a\leq|x|\leq b)}\lesssim_{a,b}R^{-2/3}.
\end{equation}
For the second term in \eqref{eq:A error vanish}, we have
\begin{align*}
    \int_{a\leq|x|\leq b}A(|h_n^{(2)}|)(x)^{3}dx
    \leq\|A(|h_n^{(2)}|)\|_{L^3(\R^2)}^{3}\lesssim\|h_n^{(2)}\|_{L^{3/2}(\R^2)}^{3},
\end{align*}
and by definition,
\begin{align*}
     \|h_n^{(2)}\|_{L^{3/2}(\R^2)}^{3/2}=\|F_n-F_\infty\|_{L^{3/2}(R^{-1}\leq|x|\leq R)}^{3/2}.
\end{align*}
Hence
\begin{equation}\label{eq:h_2 small}
    \|A(|h_n^{(2)}|)\|_{L^{3}(a\leq|x|\leq b)}\lesssim\|F_n-F_\infty\|_{L^{3/2}(R^{-1}\leq |x|\leq R)}.
\end{equation}
Now let $\eta>0$. By \eqref{eq:h_1 small} and \eqref{eq:h_3 small}, there exists $R$ depending on $\eta,a,b$ such that 
\[\|A(|h_n^{(1)}|)\|_{L^{3}(a\leq|x|\leq b)}+\|A(|h_n^{(3)}|)\|_{L^{3}(a\leq|x|\leq b)}<\frac{\eta}{2},\]
for all $n\in\mathbb{N}$. Since $\lim_{n\to \infty}\|F_n-F_\infty\|_{L^{3/2}(R^{-1}\leq|x|\leq R)}=0$ (the sequence $F_n$ is independent of $R$) and in view of \eqref{eq:h_2 small} (for the chosen, fixed value of $R$), there exists $n_0\in\mathbb{N}$ such that for all $n\geq n_0$ we have
\[\|A(|h_n^{(2)}|)\|_{L^{3}(a\leq|x|\leq b)}<\frac{\eta}{2}.\]
Therefore, \eqref{eq:A error vanish} is established.
\qed\\\\

We are ready to show that the limit $F_\infty$ in the conclusion of Theorem \ref{thm:rescaled convergence} is a critical point of $\Phi$.
\begin{theorem}\label{thm:limits are extremizers}
    Suppose $0\neq F_\infty\in L^{3/2}(\R^2)$ is as in the conclusion of Theorem \ref{thm:rescaled convergence}. Then $F_\infty\geq0$ satisfies
    \begin{equation}\label{eq:extr_id omega}
        A(F_\infty)=\Phi( F_\infty)^3\|F_\infty\|_{3/2}^{3/2} F_\infty^{1/2}\ \mbox{ a.e.}
    \end{equation}
    In particular, by Proposition \ref{prop:extremisers}, $F_\infty$ is a critical point of $\Phi$. Furthermore, we have $$\|F_\infty\|_{3/2}\geq\|f_0\|_{3/2}( 1-\varepsilon)^{2/3},$$ where $\varepsilon>0$ is the same as in the statement of Theorem \ref{thm:rescaled convergence}.
\end{theorem}
\begin{proof}
    From Theorem \ref{thm:rescaled convergence}, there exists $t_n\to \infty$ such that $$F_n:=F(t_n)\to F_\infty,\ AF_n^2\to AF_\infty^2,$$ in $L^{3/2}(S_{a,b})$ for all $0<a<b$. 
    Furthermore, from Proposition \ref{prop:extr_conv} we know that for all $0<a<b$
    \begin{align*}
        \|AF_n^2-\lambda(F_n)F_n\|_{L^{3/2}(S_{a,b})}&\leq\|AF_n^2-\lambda(F_n)F_n\|_{L^{3/2}(\R^2)}\\
        &= \|Af(t_n)^2-\lambda(f(t_n))f(t_n)\|_{L^{3/2}(\R^2)}\ \to 0,
    \end{align*} 
    as $n\to \infty$. Since $\lambda(F_n)=\lambda(f(t_n))$ is non-decreasing and bounded, it converges to some $\lambda_\infty>0$. Thus, by uniqueness of the limit in $L^{3/2}(S_{a,b})$, we get
    \[AF_\infty^2=\lambda_\infty F_\infty,\ \mbox{ a.e. in}\ S_{a,b},\]
    for all $0<a<b$. Since $F_\infty$ is non-negative, we have
    \begin{equation}\label{eq:EL for F_infty}
        AF_\infty=\lambda_\infty^{1/2}F_\infty^{1/2}.
    \end{equation}
     We shall show that 
     \begin{equation}\label{eq:lambda limit}
         \lambda_\infty^{1/2}=\Phi(F_\infty)^3\|F_\infty\|_{3/2}^{3/2}=\frac{\|\cR F_\infty\|_3^3}{\|F_\infty\|_{3/2}^{3/2}},
     \end{equation} 
     which will conclude the proof. We multiply both sides of \eqref{eq:EL for F_infty} by $F_\infty$ to get
     \begin{gather}\label{eq:EL for F_infty}
         F_\infty AF_\infty=\lambda_\infty^{1/2}F_\infty^{3/2},
     \end{gather}
     and integrate to obtain
     \[\int_{\R^2}F_\infty AF_\infty=\lambda_\infty^{1/2}\int_{\R^2} F_\infty^{3/2}.\]
     By the adjoint property and \eqref{eq:EL for F_infty}, we have
     \[\int F_\infty AF_\infty=\int\left(\cR F_\infty\right)^3,\]
     which establishes \eqref{eq:lambda limit}.\\\\
     Lastly, to show $\|F_\infty\|_{3/2}\geq 1-\varepsilon,$ we use \eqref{eq:lambda limit}, the monotonicity of $\lambda(t)$ and H\"older's inequality:
     \[
        \int (\cR f_0)^3 = \int f_0 Af_0\leq \left(\int f_0^{3/2}\right)^{1/2}\left(\int f_0^{1/2}(Af_0)^2\right)^{1/2},
     \]
     to get
     \begin{align}
         \|\cR\|^3\|F_\infty\|_{3/2}^{3/2}\geq\Phi(F_\infty)^3\|F_\infty\|_{3/2}^{3/2}=\lambda_\infty^{1/2}&\geq \lambda(0)^{1/2}\\
         &\geq \Phi(0)^3\left(\int f_0^{3/2}\right)\nonumber\\
         &\geq\|\cR\|^3\left(\int f_0^{3/2}\right)(1-\varepsilon).\nonumber
     \end{align}
\end{proof}

\section{Critical Points of $\Phi$ and the Euler--Lagrange Equation}\label{section:critical points}
We first prove Proposition \ref{prop:extremisers}, which asserts that non-negative critical points of $\Phi$ are exactly the non-trivial solutions of \eqref{eq:Euler--Lagrange 2}. In fact, we prove the following proposition, which directly implies Proposition \ref{prop:extremisers} and explicitly defines the arbitrary multiplier $\lambda$ that appears in the statement of Proposition \ref{prop:extremisers} (as well as in \cite{Christ} and \cite{Christ-Xue}) in terms of $f$.
\begin{proposition}\label{prop:extremisers2}
    A function $f\in L^{3/2}(\R^2)$ is a critical point of $\Phi$ if and only if it satisfies
    \begin{equation}
        \cR^*\left(|\cR f|\, \cR f\right)=\Phi(f)^3\|f\|_{3/2}^{3/2}|f|^{1/2}\mathrm{sgn}(f).
    \end{equation}
    In particular, $f$ is an extremiser of $\Phi$ if and only if
    \begin{equation}
        \cR^*\left(| \cR f|\,\cR f\right)=\|\cR\|^3\|f\|_{3/2}^{3/2} |f|^{1/2}\mathrm{sgn}(f).
    \end{equation}
\end{proposition}

\begin{proof}
    Since $\Phi(f)>0$ for $f\neq 0$ and every critical point is non-zero, $\Phi^3$ has the same critical points as $\Phi$. Let $\varepsilon>0$, let $\eta\in L^{3/2}$ be a smooth function such that $f+\varepsilon\eta\neq0$. We have
    \begin{gather*}
        \frac{d}{d\varepsilon}\Phi(f+\varepsilon\eta)^3=\frac{\frac{d}{d\varepsilon}\|\cR (f+\varepsilon\eta)\|_3^3\|f+\varepsilon\eta\|_{3/2}^3-\|\cR(f+\varepsilon\eta)\|_3^3\frac{d}{d\varepsilon}\|f+\varepsilon\eta\|_{3/2}^3}{\|f+\varepsilon\eta\|_{3/2}^6},
    \end{gather*}
    and therefore
     \begin{align*}         
        \|f+\varepsilon\eta\|_{3/2}^6& \frac{d}{d\varepsilon}\Phi(f+\varepsilon\eta)^3\\
        =&\, 3\int\big|\cR(f+\varepsilon\eta)\big|\cR(f+\varepsilon\eta)\frac{d}{d\varepsilon}\cR(f+\varepsilon\eta)\|f+\varepsilon\eta\|_{3/2}^3\\
        &-\|\cR (f+\varepsilon\eta)\|_3^3\, 2\int |f+\varepsilon \eta|^{3/2}\,\frac{3}{2}\int|f+\varepsilon\eta|^{1/2}\mathrm{sgn}(f+\varepsilon\eta)\frac{d}{d\varepsilon}(f+\varepsilon\eta)\\
        =&\, 3\int\big|\cR(f+\varepsilon\eta)\big| \cR(f+\varepsilon\eta) \cR(\eta)\|f+\varepsilon\eta\|_{3/2}^3\\
        &-{3}\|\cR(f+\varepsilon\eta)\|_3^3\int  |f+\varepsilon \eta|^{3/2}\int|f+\varepsilon\eta|^{1/2}\mathrm{sgn}(f+\varepsilon\eta)\eta.
    \end{align*}
    Setting $\varepsilon=0$ and simplifying yields
    \begin{align*}
        \|f+\varepsilon\eta&\|_{3/2}^{4+\frac{1}{2}}\frac{d}{d\varepsilon}\Bigg|_{\varepsilon = 0}\Phi(f+\varepsilon\eta)\\
        =&3\int\left(\cR^*\left(|\cR f|\cR f\right)\|f\|_{3/2}^{3/2}-\|\cR f\|_3^3|f|^{1/2}\mathrm{sgn}(f)\right)\eta.
    \end{align*}
    The function $f$ is a critical point of $\Phi$ if and only if
    \begin{gather}\label{eq:critical equation}
        \frac{d}{d\varepsilon}\Bigg|_{\varepsilon=0}\Phi(f+\varepsilon\eta)=0,\ \mbox{ for all smooth } \eta\in L^{3/2}\ \mbox{ such that } f+\varepsilon\eta\neq 0.
    \end{gather}
    By the Fundamental Theorem of Variational Calculus, \eqref{eq:critical equation} holds iff
    \begin{equation*}
        \cR^*\left(|\cR f|\, \cR f\right)\|f\|_{3/2}^{3/2}=\|\cR f\|_3^3|f|^{1/2}\mathrm{sgn}(f),
    \end{equation*}
    which rearranges to
    \begin{equation*}
        \cR^*\left(|\cR f|\, \cR f\right)=\Phi(f)^3\|f\|_{3/2}^{3/2}|f|^{1/2}\mathrm{sgn}(f).
    \end{equation*}

Christ \cite{Christ} showed that every $f\in\mathcal{E}$ is positive. Using this fact we get that every extremiser satisfies
\begin{gather}\label{eq:Euler--Lagrange}
    \cR^*\left[(\cR f)^2\right]=\|\cR\|^3\|f\|_{3/2}^{3/2} f^{1/2}.
\end{gather}
Lastly, we show that equation \eqref{eq:Euler--Lagrange} is not only necessary, but also a sufficient condition for a non-negative $f$ to be an extremiser. Indeed, suppose $f\geq0$ satisfies \eqref{eq:Euler--Lagrange}. We multiply both sides by $f$ and integrate to get
\[\int f\, \cR^*\left([\cR f]^2\right)=\|\cR\|^3\|f\|_{3/2}^3. \]
By the adjoint property,
\[\int f\cR^*\left([\cR f]^2\right)=\int \cR f(\cR f)^2=\int (\cR f)^3.\]
Hence, 
\[\frac{\|\cR f\|_3}{\|f\|_{3/2}}=\|\cR\|.\]

For non-negative $f$, equation \eqref{eq:Euler--Lagrange} may be rewritten as
\begin{equation}\label{eq:Euler--Lagrange 2}
    A(f)=\|\cR\|^3\|f\|_{3/2}^{3/2} f^{1/2}.
\end{equation}
By squaring both sides, multiplying by $f^{1/2}$, integrating and rearranging, we get
\[\frac{\int f^{1/2}Af^2}{\int f^{3/2}}=\|\cR\|^6\|f\|_{3/2}^3,\]
and thus for non-negative $f$, \eqref{eq:Euler--Lagrange} becomes
\[\cR^*([Rf]^2)^2=\lambda(f) f.\]
    
\end{proof}

We are now ready to prove Theorem \ref{thm:critical points}. The key result is the following.

\begin{proposition}\label{prop:A decay}
    Suppose $0\neq f\in C^\infty(\R^2) $ is positive, radial and
    \[\int_{\R^2}|x|^\kappa f(x)^{4/3}dx<\infty, \]
    for some $\kappa\in(8/15,2/3)$. Then $Af(r)\sim C(f) r^{-1}$ for $r\gtrsim 1$, where the constant $C(f)>0$ is independent of $r$.
\end{proposition}

Let $f$ be a non-negative radial critical point of $\Phi$. Then by Proposition \ref{prop:extremisers2} we have $f\sim A(f)^2$. In order to prove the assertion of Theorem \ref{thm:critical points}, we need only verify the conditions of Proposition \ref{prop:A decay}. By combining Theorem \ref{thm:Christ asymptotic} and Remark \ref{remark:C infinity}, we have $f\in C^\infty(\R^2)$, $f>0$ and $$f(x)\lesssim(1+ |x|^{2})^{-1},\qquad |\nabla f(x)|\lesssim(1+|x|^{3})^{-1},$$ 
as $|x|\to \infty$. Thus all conditions of Proposition \ref{prop:A decay} automatically hold. In fact, the last one holds for all $\kappa\in(0,2/3)$.
\begin{remark}
    The argument does not actually need the full strength of Proposition \ref{prop:A decay}, but only the lower bound $Af(x)\gtrsim |x|^{-1}$ (which is in fact easier to prove than the upper bound, as the reader may observe in the proof that follows), as the upper bound on $f$ is guaranteed by Theorem \ref{thm:Christ asymptotic} and Remark \ref{remark:C infinity}. Nonetheless, Proposition \ref{prop:A decay} is of independent interest.
\end{remark}

\begin{proof}[Proof of Proposition \ref{prop:A decay}]
    We remind the reader that if $f$ is radial then so is $Af$. Moreover, $\cR f$ is constant in the spherical variable $\alpha$. Set $g(s):=\cR f(s)^2$. For $|x|=r$ we have
    \begin{gather*}
        Af(x)=\cR^*g(x)=\int_{\bbS^1} g(x\cdot \alpha)d\alpha=4\int_0^{\pi/2} g(r\cos\theta)d\theta\\
        =4\int_0^{\pi/3}g(r\cos\theta)d\theta+4\int_{\pi/3}^{\pi/2}g(r\cos\theta)d\theta.
    \end{gather*}
    For the second integral, we change variables to get
    \begin{gather*}
       \int_{\pi/3}^{\pi/2} g(r\cos\theta)d\theta=\int_{0}^{r/2} \frac{g(s)}{\sqrt{r^2-s^2}}ds\sim\frac{1}{r}\int_0^{r/2}g(s)ds.
    \end{gather*}
    By \eqref{eq:Radon 4/3->2} and since $\cR f$ is constant in $\alpha$, we have
    \begin{align*}
        \int_0^{r/2}\cR f(s)^2ds&\leq\int_\R \cR f(s)^2 ds
       \sim \|\cR f\|_{L^2(\bbS^1\times \R)}^2\lesssim\|f\|_{L^{4/3}(\R^2)}^2.
    \end{align*}
    Moreover, for $r\geq 1$ we have
    \begin{equation*}
        \int_0^{r/2}\cR f(s)^2ds\geq\int_0^{1/2}\cR f(s)^2ds>0,
    \end{equation*}
    since $f>0$.
    \\\\
    For the first integral, we write
    \[
        J(r):=\int_{0}^{\pi/3}g(r\cos\theta)d\theta
        = \int_{r/2}^r \frac{g(s)}{\sqrt{r^2-s^2}}ds.
    \]
    We observe that $g,g'\in L^{3/2}(\R).$ Indeed, 
    \[\int g(s)^{3/2}ds=\int \cR f(s)^3ds\leq\|\cR\|^3\|f\|_{3/2}^3,\]
    and 
    \[|g'(s)|=2\cR f(s)|(\cR f)'(s)|\leq 2\cR f(s)\cR(|f'|)(s),\]
    where the last inequality is justified as follows:
    \begin{align*}
        |(\cR f)'(s)|&=\left|\int_0^\infty f'\left(\sqrt{s^2+t^2}\right)\frac{s}{\sqrt{s^2+t^2}}dt\right|\leq\int_0^\infty|f'|\left(\sqrt{s^2+t^2}\right)dt\\
        &=\int_0^\infty|\nabla f|\left(\sqrt{s^2+t^2}\right)dt=\cR(|\nabla f|)(s).
    \end{align*}
    Therefore,
    \begin{align*}
        \|g'\|_{3/2}&\sim \|\cR fR(|f'|)\|_{3/2}\leq\|\cR f\|_{3}\|R(|f'|)\|_3\\
        &\leq\|f\|_{L^{3/2}(\R^2)}\|\nabla f\|_{L^{3/2}(\R^2)}.
    \end{align*} 
    Thus, $g$ is H\"older continuous with exponent $1/3$:
    \[|g(x)-g(y)|\leq\int_y^x|g'(t)|dt\leq\|g'\|_{3/2}|x-y|^{1/3}.\]
    Moreover, for $\theta=\frac{3\kappa}{2},$ we have
    \[\int_\R s^\theta g(s)ds<\infty.\]
    Indeed, we have
    \begin{align*}
        |s|^{\theta/2}\cR f(s)&=\int_{x\cdot e_1=s}|s|^{\theta/2} f(x)d\sigma(x)=\int_{x\cdot e_1=s}|x\cdot e_1|^{\theta/2} f(x)d\sigma(x)\\
        &\leq \int_{x\cdot e_1=s}|x|^{\theta/2} f(x)d\sigma(x)=R\left(|x|^{\theta/2}f(x)\right),
    \end{align*}
    and thus
    \begin{align*}
        \|s^\theta g(s)\|_{L^1(\R)}&=\left\| s^{\theta/2}\cR f(s)\right\|_{L^2(\R)}^2\leq\left\|\cR\left( |x|^{\theta/2}f(x)\right)\right\|_{L^2(\R)}^2\\
        &\leq\|\cR\|^2\left\| |x|^{\theta/2} f(x)\right\|_{4/3}^2=\|\cR\|^2\left(\int_{\R^2}|x|^{2\theta/3}f(x)^{4/3}dx\right)^{3/2} \\
        &=\|\cR\|^2\left(\int_{\R^2}|x|^{\kappa}f(x)^{4/3}dx\right)^{3/2}<\infty.
    \end{align*}
    Thus, for $r\geq 1$ we have
    \begin{align}\label{eq:tau bound}
        \int_r^\infty g(s)ds\leq \frac{1}{r^\theta}\int_r^\infty s^\theta g(s)ds\leq \frac{1}{r^\theta}\int_0^\infty s^\theta g(s)ds\lesssim\frac{1}{r^\theta}\left(\int_{\R^2}|x|^{\kappa}f(x)^{4/3}dx\right)^{3/2}.
    \end{align}
    Now let $r\geq 1$. Since $g$ is H\"older continuous, there exists $s_r$ such that $g(s_r)=\sup_{r/2\leq s\leq r}g(s)$ and for $|s-s_r|\leq\left(\frac{g(s_r)}{2\|g'\|_{3/2}}\right)^3,$ we have
    \[g(s)\geq g(s_r)-|g(s_r)-g(s)|\geq\frac{g(s_r)}{2},\]
    and by integrating on $I_r:=\left[s_r-\left(\frac{g(s_r)}{2\|g'\|_{3/2}}\right)^3,s_r+\left(\frac{g(s_r)}{2\|g'\|_{3/2}}\right)^3\right]$, we get
    \begin{gather}\label{eq:omega tau}
        \left(\sup_{r/2\leq s\leq r}g(s)\right)^4\leq (2\|g'\|_{3/2})^3\int_{I_r}g(s)ds\leq (2\|g'\|_{3/2})^3\int_{r/2}^{r} g(s)ds.
    \end{gather}
    For $\delta<r/2$, we have 
    \begin{gather*}
        J(r)\sim\frac{1}{\sqrt{r}}\int_{r/2}^r \frac{g(s)}{\sqrt{r-s}}ds=\frac{1}{\sqrt{r}}\left[\int_{r/2}^{r-\delta}\frac{g(s)}{\sqrt{r-s}}ds+\int_{r-\delta}^r \frac{g(s)}{\sqrt{r-s}}ds \right].
    \end{gather*}
    Let $\tau(r):=\int_{r/2}^\infty g(s)ds$ and $\omega(r):=\sup_{r/2\leq s\leq r}g(s)$. We have
    \begin{gather*}
        \int_{r/2}^{r-\delta}\frac{g(s)}{\sqrt{r-s}}ds\leq\delta^{-1/2}\tau(r),
    \end{gather*} 
    and 
    \begin{gather*}
        \int_{r-\delta}^r\frac{g(s)}{\sqrt{r-s}}ds\leq\sup_{r/2\leq s\leq r} g(s)\int_{r-\delta}^r\frac{ds}{\sqrt{r-s}}=\omega(r) 2\delta^{1/2}.
    \end{gather*}
    Now we choose $\delta=\delta(r)=\frac{\tau(r)}{2\omega(r)}$ and combine the last three displays to get
    \begin{gather*}
        J(r)\lesssim\frac{1}{\sqrt{r}}\sqrt{\omega(r)\tau(r)},
    \end{gather*}
    and since by \eqref{eq:omega tau} we have $\omega(r)\lesssim\|g'\|_{3/2}^{3/4}\tau(r)^{1/4}$, \eqref{eq:tau bound} yields
    \begin{gather*}
        J(r)\lesssim\frac{1}{\sqrt{r}}\|g'\|_{3/2}^{3/8}\tau(r)^{5/8}\lesssim\|f\|_{L^{3/2}(\R^2)}^{3/8}\|\nabla f\|_{L^{3/2}(\R^2)}^{3/8} \left(\int |x|^\kappa f(x)^{4/3}dx\right)^{\frac{15}{16}}\frac{1}{r^{\frac{1}{2}+\frac{5\theta}{8}}}.
    \end{gather*}
    Provided that $\theta\geq4/5$ (equivalently $\kappa\geq8/15$), we get the desired bound.
\end{proof}

\section{Concluding Remarks}\label{sec:discussion}

\subsection{Non-Vanishing Below the Energy Threshold}\label{sec:energy threshold}
In Theorem \ref{thm:EL main}, non-vanishing of the limit relies on the initial data being $\varepsilon$-close to the extremiser-manifold $\cE$. It is not obvious to the author whether this threshold-hypothesis is sharp, or even necessary at all. A question to be addressed in future work is whether the rescaled flow converges unconditionally to a critical point of $\Phi$, or whether there exists a threshold $c$ such that if $\Phi(f_0)< c$, then every rescaling of the flow vanishes.

\subsection{The non-radial case of the Euler--Lagrange Flow}\label{sec:non-radial}
Theorem \ref{thm:EL main} crucially requires the initial data $f_0$ of \eqref{eq:EL} to be radial. It is an interesting question whether an analogue of Theorem \ref{thm:EL main} can be established without this assumption. In this setting, all other non-compact symmetries (translations, rotations and anisotropic dilations) must be eliminated by fixing an affine frame at every time $t$. It is reasonable to conjecture the following:
\begin{conjecture}
    Let $f(t)$ be the solution of \eqref{eq:EL} with non-radial initial data $f_0\in W^{1,3/2}(\R^2)$. Then there exist invertible affine maps $\phi(t),\ t\geq0 $ such that $$|\det \phi(t)|^{4/3}f(\phi(t)\cdot,t)$$ converges to a critical point $f_\infty$ of $\Phi$.
\end{conjecture}

\subsection{Higher-Dimensional Radon-Transform Monotonicity}
It is also natural to ask whether the Radon-flow monotonicity results of this paper can be extended in higher-dimensional Radon transforms. Given that Christ \cite{Christ} characterised the extremisers of the Radon transform in every dimension, we expect the results to generalise in a fairly straightforward way. 

\appendix

\section{}

\subsection{Proof of Lemma \ref{le:barbalat}}\label{ap:barbalat}

Let $\varepsilon>0$. Since $\xi\in L^1(\R_+)$, by the Lebesgue differentiation theorem the integral $\int_0^t |\xi(s)|ds$ is absolutely continuous. Hence, there exists $\delta_1>0$ such that if $|t_1-t_2|\leq \delta_1$, then 
\[\int_{t_1}^{t_2} |\xi(t)|dt<\varepsilon/2.\]
Since $\nu\in L^\infty(\R)$, if $|t_1-t_2|\leq\frac{\varepsilon}{2\|\nu\|_\infty}$ we have
\[\int_{t_1}^{t_2} |\nu(s)|ds\leq\|\nu\|_\infty|t_1-t_2|\leq\varepsilon/2. \]
By taking $\delta=\min\{\delta_1,\varepsilon/\|\nu\|_\infty\}$ and $t_1,t_2$ such that $|t_1-t_2|\leq \delta,$ we have
\begin{align*}
    |u(t_2)-u(t_1)|=\left|\int_{t_1}^{t_2}u'(t)dt\right|\leq\int_{t_1}^{t_2}|\xi(t)|dt+\int_{t_1}^{t_2}|\nu(t)|dt\leq \varepsilon.
\end{align*}
Thus, $u$ is uniformly continuous. The second claim follows from the fact that any integrable and uniformly continuous function vanishes at infinity (generalised Barbalat's lemma). \qed

\subsection{Fractional Leibniz Rule}\label{ap:Leibniz}
We state the extension of the Kato-Ponce inequality proven by Grafakos and Oh \cite{Grafakos-Oh}.
\begin{theorem}[Theorem 1, Grafakos-Oh \cite{Grafakos-Oh}]\label{thm:Leibniz}
    Let $\frac{1}{2}<r<\infty,\ 1<p_1,p_2,q_1,q_2\leq \infty$ satisfy $\frac{1}{r}=\frac{1}{p_1}+\frac{1}{q_1}=\frac{1}{p_2}+\frac{1}{q_2}$. Given $s>\max\left\{0,\frac{n}{r}-n\right\}$ or $s\in 2\mathbb{N}$, there exists $C=C(n,s,p_1,p_2,q_1,q_2)<\infty$ such that for all $f,g\in \cS(\R^n)$ we have
    \begin{align}
        &\nonumber\left\|(-\Delta)^{s/2}(fg)\right\|_{L^r(\R^n)}\\
        &\qquad\leq C\left(\left\|(-\Delta)^{s/2}f\right\|_{L^{p_1}(\R^n)}\|g\|_{L^{q_1}(\R^n)}+\|f\|_{L^{p_2}(\R^n)}\left\|(-\Delta)^{s/2}g\right\|_{L^{q_2}(\R^n)} \right).
    \end{align}
\end{theorem}

\subsection{Compact Embeddings of Triebel-Lizorkin Spaces}\label{ap:Triebel-Lizorkin}
The literature on compact embeddings of function spaces, starting with the classical result of Rellich and Kondrachov (which may be found as Theorem $9.3$ in \cite{Brezis}) is vast. Here we use Triebel's theorem for compact embeddings of $F^s_{p,r}(\Omega)$ for bounded Lipschitz domains.
\begin{theorem}[Triebel \cite{Triebel}, Theorem $3.2.3$]\label{thm:Triebel general}
Let $\Omega\subset\R^n$ be a bounded Lipschitz domain. The identity map $i:F^{s_1}_{p_1,r_1}(\Omega)\to F^{s_2}_{p_2,r_2}(\Omega)$ is compact (maps bounded sequences to precompact ones) if and only if
\[s_1\geq s_2\qquad \mbox{ and }\qquad s_1+\frac{n}{p_1}>s_2+\frac{n}{p_2}.\]
\end{theorem}
As Triebel remarks (and is indeed obvious from the definitions) in Paragraph $3.2.2.$, for any compact interval $I$, the space $F^s_{p,2}(I)$ coincides with the Bessel-potential space $H^{s,p}(I)$, defined as the quotient spaces of $F^s_{p,2}(\R)$ and $H^{s,p}(\R)$ respectively modulo the restriction map $U\mapsto U|_I$.

\subsection{The Hardy--Littlewood--Sobolev Inequality}\label{ap:HLS}
There are multiple references for the Hardy--Littlewood--Inequality. Here we use the one in \cite{Stein-Singular}, stated in terms of Riesz Potentials.
\begin{theorem}[Stein \cite{Stein-Singular}, Chapter $\mathrm{V}$, Theorem 1]
We have
\begin{equation}\label{eq:HLS}
    \|I_\alpha f\|_{L^q(\R^n)}\leq C_{\alpha,p,q}\|f\|_{L^p(\R^n)},
\end{equation}
for $0<\alpha<n,\ 1<p<q<\infty,\ \frac{1}{q}=\frac{1}{p}-\frac{\alpha}{n}$.
    
\end{theorem}

\subsection{A bound for $\cR$}\label{ap:Radon bound}

Using \eqref{eq:fourier slice} and the Hardy--Littlewood--Sobolev inequality, we prove the following
\begin{proposition}
    For all $f\in C_0^\infty(\R^2)$ we have
    \begin{equation}\label{eq:Radon 4/3->2}
        \|\cR f\|_{L^{2}(\bbS^1\times\R)}\leq 2{ C_{\mathrm{HLS}}}{\pi^{1/2}}\|f\|_{L^{4/3}(\R^2)},
    \end{equation}
    where $C_{\mathrm{HLS}} $ is the sharp constant in \eqref{eq:HLS}.
\end{proposition}
\begin{proof}
    By Plancherel's identity and \eqref{eq:fourier slice}, we have
    \begin{align*}
        \int_{\bbS^1}\int_{\R}|\cR f(\alpha,s)|^2 ds\, d\alpha&= \int_{\bbS^1} \int_{ \R}|\cR_\alpha f(s)|^2\,ds\,d\alpha    =\int_{\bbS^1} \int_{\R}\left|\widehat{\cR_\alpha f}(\sigma)\right|^2\,d\sigma\, d\alpha\\
        &={2\pi}\int_{\bbS^1}\int_{\R}|(2\pi\sigma)^{-1/2}\widehat{f}(\sigma\alpha)|^2|\sigma|d\sigma\, d\alpha\\
        &=4{\pi}\int_{\bbS^1}\int_0^\infty|(2\pi\sigma)^{-1/2}\widehat{f}(\sigma\alpha)|^2\sigma \,d\sigma\, d\alpha\\
        &=4{\pi}\int_{\R^2}|2\pi\xi|^{-1}|\widehat{f}(\xi)|^2\,d\xi=4{\pi}\left\|(-\Delta)^{-1/4}f\right\|_{L^{2}(\R^2)}^2.
    \end{align*}
    By \eqref{eq:HLS} for $n=2,\alpha=1/2,q=2,p=4/3$, we have
    \[\left\|(-\Delta)^{-1/4}f\right\|_{L^{2}(\R^2)}\leq C_{\mathrm{HLS}}\|f\|_{L^{4/3}(\R^2)}. \]
\end{proof}

\subsection{Gronwall's Lemma}\label{ap:Gronwall}
Gronwall's Lemma \cite{Gronwall} --- in its simplest possible formulation --- asserts the following:
\begin{lemma}\label{lemma:Gronwall}
Suppose $u(t),\beta(t)$ are continuous functions and $\alpha(t)\geq0$ is a non-decreasing function such that
\[u(t)\leq \alpha(t)+\int_0^t \beta(s)u(s)ds,\]
for all $t\geq0$. Then 
\[u(t)\leq\alpha(t) e^{\int_0^t \beta(s)ds}, \]
for all $t\geq 0$.
\end{lemma}

\printbibliography

\end{document}